\documentclass[10pt,reqno]{amsart}

\usepackage[left=2.5cm, right=2.5cm, top=2cm]{geometry}

\usepackage[english]{babel}

 \usepackage[foot]{amsaddr}

\usepackage{amsmath,amssymb,amsthm,amsfonts,mathrsfs,amsxtra,bbm,comment}

\usepackage{tikz-cd}

\usepackage[hyperfootnotes=true,
        pdffitwindow=true,
        plainpages=false,
        pdfpagelabels=true,
        pdfpagemode=UseOutlines,
        pdfpagelayout=SinglePage,
        hyperindex,]{hyperref}
        
\usepackage[shortlabels]{enumitem}

\numberwithin{equation}{section}

\theoremstyle{plain}
	\newtheorem{theorem}{Theorem}
	\newtheorem{lemma}[theorem]{Lemma}
	\newtheorem{corollary}[theorem]{Corollary }
	
\theoremstyle{definition}

	\newtheorem{remark}[theorem]{Remark}

\newcommand{\N}{\mathbb N}

\newcommand{\R}{\mathbb R}
\newcommand{\C}{\mathbb C}

\newcommand{\U}{\mathrm U}

\newcommand{\Herm}{\mathrm{Herm}}

\newcommand{\dv}{\mathrm{d}}
\newcommand{\tr}{\mathrm{Tr}\,}
\newcommand{\eig}{\mathrm{eig}}
\newcommand{\diag}{\mathrm{diag}}

\newcommand{\sgn}{\mathrm{sgn}}
\newcommand{\la}{\left\langle}
\newcommand{\ra}{\right\rangle}

\newcommand{\ind}{\mathbbm{1}}
\newcommand{\re}{\mathrm{Re}\,}
\newcommand{\im}{\mathrm{Im}\,}

\usepackage{lipsum}
\newcommand{\mylongtitle}[1]{%
	\ifodd\value{page}%
	\protect\parbox{0.97\linewidth}{#1}\hfill%
	\else%
	\hfill\protect\parbox{0.97\linewidth}{#1}%
	\fi%
}

\begin{document}

\title[\mylongtitle{\centering A rate of convergence when generating stable invariant \\ Hermitian random matrix ensembles}]{A rate of convergence when generating stable invariant Hermitian random matrix ensembles}
	
\author[Mario Kieburg]{Mario Kieburg}
\address[MK]{School of Mathematics and Statistics, University of Melbourne, 813 Swanston Street, Parkville, Melbourne VIC 3010, Australia}
\email[MK]{(MK) m.kieburg@unimelb.edu.au}

\author[Jiyuan Zhang]{Jiyuan Zhang}
\address[JZ]{Department of Mathematics, KU Leuven, Celestijnenlaan 200B, B-3001 Leuven, Belgium}
\email[JZ]{(JZ) jiyuanzhang.ms@gmail.com}

\date{\today}

\begin{abstract}
Recently, we have classified Hermitian random matrix ensembles that are invariant under the conjugate action of the unitary group and stable with respect to matrix addition. Apart from a scaling and a shift, the whole information of such an ensemble is encoded in the stability exponent determining the ``heaviness'' of the tail and the spectral measure that describes the anisotropy of the probability distribution. In the present work, we address the question how these ensembles can be generated by the knowledge of the latter two quantities. We consider a sum of a specific construction of identically and independently distributed random matrices that are based on Haar distributed unitary matrices and a stable random vectors. For this construction, we derive the rate of convergence in the supremum norm and show that this rate is optimal in the class of all stable invariant random matrices for a fixed stability exponent. As a consequence we also give the rate of convergence in the total variation distance.
\end{abstract}

\maketitle

\section{Introduction}

Heavy tailed random matrix ensembles naturally appear in several applications ranging from statistics to physics, engineering and machine learning, e.g., see~\cite{BJNPZ2001,MS2003,BGW2006,BCP2008,AFV2010,BT2012,MSG2014,Kanazawa2016,OT2017, Minsker2018,MM2018,MM2019,Heiny2019a,Shin}. Therefore, it is a very intriguing question to study also random matrix ensembles that are stable with respect to matrix addition, see~\eqref{stab.cond}. The simplest class from its construction are heavy-tailed Wigner matrices where the matrix entries are identically and independently distributed heavy tailed random variables. They were heavily studied in physics and mathematics~\cite{CB1994,BJNPZ2001,Heiny2019a,Soshnikov2004,BBP2007,BJNPZ2007,BAG2008,ABAP2009,BCHJ2021,Vershynin2012, BGGM2014,BGM2016,HM2017,BG2017,Male2017,GLPTJ2017}. The stable version of these matrices are then readily found by choosing ordinary stable real random variables as the real independent matrix entries.

Indeed heavy tailed Wigner matrices exhibit interesting statistics, especially when the matrix dimension tends to infinity. For instance the statistics of the largest eigenvalues may converge to the Poisson statistics~\cite{Soshnikov2004,BBP2007,ABAP2009,TBT2016}  and the eigenvectors perform a transition from being localised~\cite{BT2012,CB1994,TBT2016} for the largest eigenvalues to being delocalised~\cite{BG2017} in the bulk. 

Going away from Wigner matrices and introducing group invariance such as unitarily invariant random matrices~\cite{Tierz,AFV2010,Kanazawa2016,BJNPZ2007,BJJNPZ2002,AV2008,AMAV2009,CM2009,Thomas, Tdistribution,KieMont},
meaning matrices that are invariant under the conjugate action of the unitary group (see~\eqref{uni.inv}), may change the picture. While also Poisson statistics has been observed~\cite{KieMont} for the largest eigenvalues in the limit of large matrix dimensions, the eigenvectors remain always delocalised. This is due to unitary invariance as then the eigenvectors are distributed with respect to the Haar measure of the unitary group. To understand this and other peculiarities for invariant stable random matrices, we have started an approach from finite matrix dimension. In  a recent work~\cite{KZ21}, we have classified all unitarily invariant stable random matrix ensembles on the Hermitian matrices $\Herm(N)$ and we have identified their domains of attraction meaning we have proven central limit theorems. In the next step, which is the present work about, we would like to show how those invariant stable random matrix ensembles can be created via an approximation and how fast the convergence of this approximation is.

\subsection{Motivation}

It is well known in probability theory that for any real random variable $X$ having a finite second moment, the standardised sum of identical and independent copies of $X$
\begin{equation}\label{iid_sum}
	Y_m=\frac{1}{m^{1/2}}\sum_{j=1}^m(X_j-\mathbb E X_j)
\end{equation}
tend to the standard normal distribution in the weak sense, where $\mathbb E f(X)$ denotes the expectation value. This result is known as the classical central limit theorem. There are a lot of analysis on this convergence in the literature, particularly on the rate of convergence---for example the Berry-Esseen theorem (see e.g.~\cite[Ch. XVI Sec. 3]{Fellerbook}) states the difference between the Kolmogorov distances of the sum and the limit. The classical central limit theorem can also be extended to the multivariate case, as well as results regarding rates of convergence; see~e.g.~\cite{Fellerbook,GKCbook}.

When the condition that the random variable has finite second moment is dropped, the limit of the sum~\eqref{iid_sum} may not always be a normal distribution; in the worse case when $\mathbb E X$ also does not exist, the sum is not even well-defined. For those heavy-tailed distributions, one instead considers the sum
\begin{equation}\label{iid_sum2}
Y_m=\frac{1}{B_m}\sum_{j=1}^m X_j-A_m
\end{equation}
where $A_m\in\mathbb{R}$ and $B_m\in\mathbb{R}_+$ are sequences of real numbers. If there exist two such sequences $\{A_m\}$ and $\{B_m\}$ that the sum~\eqref{iid_sum2} tends to a random variable in the weak sense, we call the corresponding limiting distribution the stable distribution, and $X$ is said to belong to the domain of attraction of this limiting distribution. This is the generalised version of the classical central limit theorem; see e.g.~\cite{Rvaceva62,Shimura90,GKCbook} for further details.

The generalised central limit theorem can also be extended to the multivariate case, where a multivariate stable distribution is introduced; see e.g.~\cite{STbook} for details. A multivariate stable distribution $\R^d$ can be fully classified by a quadruplet:
\begin{enumerate}[(1)]
	\item the stability parameter $\alpha\in(0,2]$ which determines the tail behaviour of the density;
	\item the spectral measure $H$ which is a probability measure defined on the unit sphere and encodes the angular distribution of the random vector;
	\item the scaling parameter $\sigma\in(0,\infty)$ which fixes the width of the distribution;
	\item the shift parameter $\mu\in\R^d$ which is related to the median of the distribution.
\end{enumerate}
Under this classification, we denote a stable distribution as $S(\alpha,\sigma H,\mu)$. Note that this has a continuous probability density function on a vector space, and so we also denote this function with variable $X$ as $S(\alpha,\sigma H,\mu;X)$. The product $\sigma H$ indicates that these two objects only appear in this combination. The distribution $S(\alpha,\sigma H,\mu)$  has the following characteristic function
\begin{equation}\label{char.stable.vec}
\mathbb E\exp(i\xi^\top X)=\displaystyle
\exp\left(-\sigma\int_{\mathbb S^{d-1}}H(\dv r) \nu_\alpha(r^\top\xi)+i\xi^\top\mu\right)
\end{equation}
where
\begin{equation}\label{nualpha}
\nu_\alpha(u):=\begin{cases}\displaystyle
|u|^\alpha\left(1-i\,\sgn(u)\tan\frac{\pi\alpha}{2}\right)=\frac{(-iu)^\alpha}{\cos(\pi\alpha/2)
},&\alpha\in(0,1)\cup(1,2],\\\displaystyle
|u|\left(1+\frac{2i}{\pi}\sgn(u)\log|u|\right)=\frac{2iu}{\pi}\log(-iu),&\alpha= 1.
\end{cases}
\end{equation}
We employ the principal value of the complex exponent and the complex logarithm, meaning the branch cut is along the negative real line.

As all stable distributions can be parametrised by $\alpha$, $\sigma$, $H$, and $\mu$, one is interested in how to sample such a distribution with those given parameters. Two approaches of sampling are proposed and compared in the paper by Davydov and Nagaev~\cite{DN02}. The first one is based on an approximation on its spectral measure, while the second one is based on the central limit theorem (also see~\cite{BNR93} for a similar approach). In the second approach they use a sum of a large number of independent and identical heavy-tailed distributions (to be specific a Pareto-type distribution is chosen) as the approximation to stable distributions. The efficiency of this method is presented as two rate-of-convergence results of the generalised central limit theorem.

In a previous work of the authors~\cite{KZ21}, the stable distributions for unitarily invariant Hermitian random matrices are introduced. They are special cases of multivariate stable distributions, with the extra invariance condition that when applying a unitary similarity transformation the distribution is unchanged. Such a distribution is called stable invariant ensemble. A classification result of stable invariant ensembles has been also given, which will be restated in Sec.~\ref{s.pre}.

In the present work, we aim at giving an efficient way of sampling those stable invariant ensembles. It is natural to use both of the Davydov-Nagaev methods to do the sampling. However, we found that one can furthermore utilize the invariance properties of those ensembles to improve the second sampling method. In Sec.~\ref{s.pre} we present improved versions of~\cite[Thm 3.2 \& Thm 3.3]{DN02}. 

The rest of our work is structured as follows. In Sec.~\ref{s.pre} we will firstly re-state the classification result of stable invariant ensembles introduced in~\cite{KZ21}. Then we present three main results: Theorem~\ref{trace_traceless} states the regularity of stable invariant ensembles; Theorem~\ref{thm2} suggests an approximation and states its rate of convergence in terms of the supremum norm; and Corollary~\ref{col3} states the rate of convergence of the total variation distance. After notations used in this paper are introduced in Sec.~\ref{s.notations}, in Secs.~\ref{sec:reg},~\ref{s.proof}, and~\ref{s.proof3} all the proofs for Theorem~\ref{trace_traceless} and Theorem~\ref{thm2} are presented. We conclude our discussions in Sec.~\ref{sec:conclusio} where Corollary~\ref{col3} will be given.

\subsection{Preliminaries and results}\label{s.pre}

We denote the set of $N\times N$ Hermitian matrices by $\Herm(N)$; we may also identify it with $\R^{N^2}$. Moreover, $\U(N)$ denotes the unitary group. We recall that a distribution for the random matrix $X\in\Herm(N)$ is stable if there is a $B>0$ and an $A\in\Herm(N)$ such that the sum of two independent copies $X_1$ and $X_2$ of $X$ is equal in distribution like
\begin{equation}\label{stab.cond}
X\stackrel{\dv}{=}\frac{1}{B}(X_1+X_2-A),
\end{equation} 
where $\stackrel{\dv}{=}$ denotes that both sides share the same distribution.
Additionally, we say that a stable distribution on $\Herm(N)$ is unitarily invariant if
\begin{equation}\label{uni.inv}
	X\stackrel{\dv}{=}UXU^\dagger, \forall\, U\in\U(N)
\end{equation}
with the Hermitian adjoint $U^\dagger$ of $U$. It can be fully classified by the same set of four parameters as for random vectors, with the subtle differences that
\begin{enumerate}[(1)]
	\item the spectral measure $H$  is now defined on the unit sphere $\mathbb S_{\Herm(N)}\subset\Herm(N)$ satisfying $H(\dv X)=H(U\dv XU^\dagger)$ for all $U\in\U(N)$, and
	\item the shift parameter $\mu\in\R$  is only a real number.
\end{enumerate}
We will denote such a stable invariant ensembles by $S(\alpha,\sigma H,\mu)$, analogously to the multivariate case, and the corresponding characteristic function is
\begin{equation}\label{char.stable.mat}
\mathbb E\exp(i\tr SX)=\displaystyle
\exp\left(-\sigma\int_{\mathbb S_{\Herm(N)}}H(\dv R) \nu_\alpha(\tr RS)+i\mu\tr S\right).
\end{equation}
The reason why the shift $\mu$ is only a real scalar and not a matrix itself is the unitarily invariance. This invariance does not only imply the unitarily invariance of the spectral measure $H$ but also forces us to have an invariant for the linear shift $\tr (\mu S)$ in the logarithm of the characteristic function. The only linear term that satisfies $\tr (\mu S)=\tr (\mu USU^\dagger)$ or all $U\in\U(N)$ and $S\in\Herm(N)$ is a multiple of the identity matrix because of Schur's lemma.

Our first result is a statement about the regularity of the distribution of any stable invariant ensemble. It is proven to the end of Sec.~\ref{sec:reg}. Let us underline that we exclude the Dirac measure at a specific Hermitian matrix as a stable distribution in the ensuing sections as the conclusions for this case are trivial.

\begin{theorem}[Regularity Property]\label{trace_traceless}
	Let $X\in\Herm(N)$ be a non-zero random matrix drawn from an invariant stable ensemble with probability distribution $F$. Then exactly one of the following three statements is true:
	\begin{enumerate}
		\item $F$ is a stable distribution on an $N^2$-dimensional subspace of $\Herm(N)$, and it is bounded and continuous on $\Herm(N)$;
		\item  There is a $t\in \mathbb{R}$ such that $F$ is a stable distribution with a support which is restricted on the subspace of $\Herm(N)$ of fixed trace Hermitian matrices
		\begin{equation}
		\Herm_t(N):=\left\{X\in\Herm(N): \tr X=t\right\},
		\end{equation}
		and it is bounded and continuous thereon;
		\item $F$ is a stable distribution with a support restricted to the $1$-dimensional subspace $\{xI_N:x\in\mathbb{R}\}\subset\Herm(N)$, while the traceless part of the random variable $X\in\Herm(N)$ always vanishes.
	\end{enumerate}
\end{theorem}

The result is rather intuitive from a representation theoretic perspective. The two irreducible representations of the conjugate action of $\U(N)$ on $\Herm(N)$ is the traceless part which is equal to the fundamental representation of the Lie algebra ${\rm SU}(N)$ apart from a global factor of $i$ and the trivial representation given by the multiplicative of the identity matrix $I_N$. Certainly the two components may have independently a different stability exponent $\alpha$. Yet, the stability condition~\eqref{stab.cond} essentially implies that only one of the two components can be probabilistic when they do not share the same $\alpha$.

Once we have proven Theorem~\ref{trace_traceless}, we will only consider the first two cases in the ensuing discussion where $X$ is not proportional to the identity matrix. As every fixed trace matrix can be shifted by a multiple of the identity matrix to a traceless matrix, we also restrict ourselves to $\Herm_0(N)$ without loss of generality. 

To develop the second main theorem of this paper, we employ the following notations.
\begin{remark}[Notations and Assumptions]\label{rem:notation}\
\begin{enumerate}
\item	First of all we assume $N>1$, since $N=1$ corresponds to the univariate case which is well understood.
\item	For convenience let us use a unified notation $W$ for either $\Herm(N)$ of $\Herm_0(N)$ in the above first two cases.
\item The spectral measure $H$ is distributed on $\mathbb S_W$, the unit sphere of $\Herm(N)$; we furthermore assume it is a \textit{Borel measure} with respect to the standard topology  on $\mathbb S_W$ induced by the Euclidean norm in $\Herm(N)$.
\item  The eigenvalue part of $H\in\Herm(N)$ is denoted by $H_{\eig}$.
\item	Let $X_j=\diag(\widetilde{X}_j)$ be diagonal matrices with entries given by the vector entries of $\widetilde{X}_j\in\mathbb{R}^N$ which are independent and identical copies of a random vector drawn from $S(\alpha,H_\eig,0)$ and let $Y\in\Herm(N)$ follow $S(\alpha,H,0)$. Note that the former is an $N$-dimensional stable distribution on the real diagonal matrices whose set can be identified with $\mathbb{R}^N$ while the latter is a stable invariant ensemble in $\Herm(N)$ which should not be confused. The choices $\mu=0$ and $\sigma=1$ are not restrictions of generality as $\mu$ only shifts the random matrix $Y$ and the diagonal random matrices $X_j$ by the fixed matrix $\mu I_N$ and $\sigma$ only rescales those.
\item	Let  $U_j\in\U(N)$ be independently and identically distributed copies of unitary matrices that should be Haar distributed.
\item	We assume that all $X_j$ and $U_j$ are also statistically independent.
\end{enumerate}
\end{remark}
These notations and assumptions will be used throughout the present work such as in our second main result.

\begin{theorem}[Rate of convergence]\label{thm2}
	With the notations and assumptions in Remark~\ref{rem:notation}, we denote by $p_m$ the probability density function of the standardised sum 
	\begin{equation}
	Y_m:=\frac{1}{m^{1/\alpha}}\sum_{j=1}^mU_jX_jU_j^\dagger-\frac{2t_0\log m}{\pi}I_N\ind_{\alpha=1}
	\end{equation}
	with $t_0:=\frac{1}{N}\int_{\mathbb S^{N-1}}H_\eig(\dv t)\tr (t)$, $I_N$ being the $N\times N$ identity matrix, and $\ind_{\alpha=1}$ being the indicator function that is $1$ when $\alpha=1$ and zero otherwise. Furthermore, the probability density function of $Y$ is denoted by $p_\infty$, which is distributed according to $S(\alpha,H,0)$. Then,
	\begin{enumerate}
		\item The density $p_1$ is bounded and continuous on $W$ if and only if $p_\infty$ is bounded and continuous on $W$, and this implies that for any $m\in\N$, the distribution of $p_m$ is also bounded and continuous;
		\item there exists a positive constant $C$ such that the bound
		\begin{equation}\label{thm.bound}
		\sup_{X\in W}|p_m(X)-p_\infty(X)|\le \frac{C}{m}
		\end{equation}
		holds for any $m\geq1$;
		\item furthermore the rate $O(m^{-1})$ for the dominating term is optimal for all $\alpha\in(0,2]$, in the set of all spectral measures $H_{\rm eig}$ on $\mathbb{R}^N$.
	\end{enumerate}
\end{theorem}

\begin{remark}\ 
	\begin{enumerate}
		\item The third part about optimality means explicitly that for every sequence of positive numbers $\{a_m\}_{m\in\mathbb{N}}\subset\mathbb{R}_+$ satisfying the bound $\sup_{X\in W}|p_m(X)-p_\infty(X)|\le a_m$ for all $m\in\mathbb{N}$ and spectral measures $H_{\rm eig}$ on $\mathbb{R}^N$ it must be 
		\begin{equation}
			\lim_{m\to\infty}\frac{1}{a_m m}<\infty.
		\end{equation}
		The limit actually implies that $a_m$ does not converge faster to $0$ than the sequence $\{1/m\}_{m\in\mathbb{N}}$.
		
		\item We utlise a similar technique to~\cite{DN02} on estimating the integral for proving statement (2). Compared with the rate of convergence $O(m^{-\min(1,2/\alpha-1)})$ obtained in~\cite{DN02}, our rate is the same when $\alpha\le 1$ and sharper when $\alpha>1$.  In addition, the results in~\cite{DN02} were only shown for symmetric distributions, while we believe that this restriction can be removed when using some of the techniques in the present work.
		
		\item The novelties of this theorem lie in parts (1) and (3). For part (1) one has to investigate the regularities of the matrix versions of heavy-tailed distributions. For part (3), a standard strategy which is also adopted in~\cite{DN02} is to find a specific distribution whose rate of conergence can be exactly computed, in order to obtain a lower bound of the distance. However, it is not easy to find such a distribution in our case.  In the proof of statement (3) we instead find a way to get around it by giving a lower bound of the exact rate for the specifically chosen distirbution.
		
		\item We would like to point out here that the constant $C$, which is encountered in our proof, is far from being optimal, the rate $1/m$ is though. Moreover, we are certain that the next term in a large $m$ expansion can be improved to $O(m^{-2})$ with the same strategy applied in~\cite{DN02} and in our work. 
	\end{enumerate}
\end{remark}

Part (1) is proven at the end of Sec.~\ref{sec:reg}. The proof of part (2) is given in Sec.~\ref{s.proof} and the proof of part (3) about the optimality of the rate of convergence is given in Sec.~\ref{s.proof3}.

\subsection{Some notations}\label{s.notations}

In the ensuing sections we also make use of several specific notations, which are all presented here for the readers' convenience. 

For a function $f:\U(N)\mapsto \C$ defined on the unitary group,
we use the average notation
\begin{equation}
	\la f(U)\ra:=\int_{\U(N)}\mu_{\rm Haar}(\dv U)\, f(U)
\end{equation}
for the integral over the unitary group with respect to the normalised Haar measure $\mu_{\rm Haar}(\dv U)$. This allows to easily distinguish whether we take the expectation value over the random matrices $X_j$ or over the unitary matrices $U_j$.

For a vector $r=(r_1,\ldots,r_N)$, we introduce the notation $\diag(r)$ for a diagonal matrix with the diagonal entries being $r_1,\ldots,r_N$. Given a spectral measure $H$ with eigenvalue part $H_\eig$, we define the function $w_\alpha:\Herm(N)\mapsto \C$ by
\begin{equation}\label{def.w-alpha}
w_\alpha(X):=\int_{\mathbb S^{N-1}}H_\eig(\dv r) \nu_\alpha(\tr \diag(r)X)
\end{equation}
with $\nu_\alpha$ as in~\eqref{nualpha}, which will have a regular appearance in our computations.
Another important quantity is the variance of $w_\alpha(USU^\dagger)$ with respect to the unitary matrix $U\in\U(N)$  and a fixed matrix $S\in W$, i.e.,
\begin{equation}\label{def.variance}
	v(S):=\la w_\alpha(USU^\dagger)^2\ra-\la w_\alpha(USU^\dagger)\ra^2.
\end{equation}
This quantity relates to the constant $C$ appearing in~\eqref{thm.bound}. We would like to point out that $v(S)$ can be complex in general, see~\eqref{nualpha}. Also associated with the quantity $w_\alpha$ are the two constants
\begin{equation}
b_H:=\inf_{\tr S^2=1}\la\re w_\alpha(USU^\dagger)\ra \quad{\rm and}\quad c_H:=\exp(-b_H/m).
\end{equation}
Both constants are bounded in $S\in\mathbb{S}_{\Herm(N)}$ as we will see.
They will be introduced again in Lemma~\ref{walpha} and used throughout the proof of Theorem~\ref{thm2}.

As we will use the harmonic analysis approach in our proofs, we summarised some conventions here. For an $L^{1}(W)$ function $f$ we denote the Fourier transform of $f$ by
\begin{equation}
	\hat f(S):=\int_{W}\dv X\,f(X)e^{i\tr XS}.
\end{equation}
In the case $W=\Herm(N)$, $\dv X$ is the product of differentials of the Lebesgue measure of independent entries,
\begin{equation}
	\dv X=\prod_{j=1}^N\dv x_{j,j}\prod_{j<k}\dv x_{j,k}^{(R)}\dv x_{j,k}^{(I)},
\end{equation}
while in the case $W=\Herm_0(N)$, $\dv X$ is the induced measure from the previous one on $\Herm(N)$ by restricting the trace of $X$ to zero. For a random variable $X$ with a probability density function $f$ defined on $W$, we denote the characteristic function  by
\begin{equation}
	\mathbb E_X\exp(i\tr XS)=\hat f(X),
\end{equation}
as well. The following lemma states what the characteristic functions of the two distributions $p_m$ and $p_\infty$ in Theorem~\ref{thm2} are.

\begin{lemma}

With $p_m$ and $p_\infty$ defined in Theorem~\ref{thm2}, in both $\alpha\ne 1$ and $\alpha=1$ cases one has
\begin{equation}
\hat p_m(S)=\la\exp\left(-\frac{w_\alpha(USU^\dagger)}{m}\right)\ra^m\quad{\rm and}\quad \hat p_\infty(S)=\exp\la-w_\alpha(USU^\dagger)\ra.
\end{equation} 

\end{lemma}

\begin{proof}
	The expression for $\hat p_\infty$ can be obtained from the characteristic function of a stable multivariate distribution, see~\eqref{char.stable.mat} for $\mu=0$. To find the expression in our claim we note that
	\begin{equation}
	\int_{\mathbb S_{\Herm(N)}}H(\dv R) \nu_\alpha(\tr RS)=\left\langle\int_{\mathbb S^{N-1}}H_\eig(\dv r) \nu_\alpha(\tr U^\dagger\diag(r)US)\right\rangle=\la w_\alpha(USU^\dagger)\ra
	\end{equation}
	which follows from the eigenvalue decomposition $R=U^\dagger\diag(r)U$. Since the spectral measure is also unitarily invariant the  Haar measure with respect to $U$ factorises from the one of the measure $H_\eig(\dv r)$ for the eigenvalues.
	
	Regarding the equation for $\hat p_m$ we note that
	\begin{equation}
	\begin{split}
	&\mathbb E\exp\left(im^{-1/\alpha}\sum_{j=1}^N\tr U_jX_jU_j^\dagger S-i\frac{2t_0\log m}{\pi}\tr S\,\ind_{\alpha=1}\right)\\
	=&\left(\prod_{j=1}^m\la\mathbb E_{X_j}\exp\left(i{m^{-1/\alpha}}\tr U_jX_jU_j^\dagger S\right)\ra\right)\exp\left(-i\frac{2t_0\log m}{\pi}\tr S\,\ind_{\alpha=1}\right)\\
	=&\la\exp \left(-w_\alpha(m^{-1/\alpha}U^\dagger SU)\right)\ra^m\exp\left(-i\frac{2t_0\log m}{\pi}\tr S\,\ind_{\alpha=1}\right).
	\end{split}
	\end{equation}
	In the $\alpha\ne 1$ case, one has $w_\alpha(m^{-1/\alpha}U^\dagger SU)=m^{-1}w_\alpha(U^\dagger SU)$ which reclaims the result. When $\alpha=1$, one has instead
	\begin{equation}
		\nu_1(m^{-1}\tr \diag(r)U^\dagger SU)=m^{-1}\nu_1(\tr \diag(r)U^\dagger SU)-\frac{2i}{\pi}\frac{\log m}{m}\,\tr \diag(r)U^\dagger SU.
	\end{equation}
	Taking the $\dv r$ integral with the measure $H_\eig$ on both sides and noticing that
	\begin{equation}
		\int_{\mathbb S_W}H_\eig(\dv r)\tr \diag(r)U_j^\dagger SU_j=t_0\tr S
	\end{equation}
	due to the permutation symmetry of the measure $H_\eig$, we have
	\begin{equation}
		w_1(m^{-1}\tr \diag(r)U_j^\dagger SU_j)=m^{-1}w_1(\tr \diag(r)U_j^\dagger SU_j)-i\frac{2t_0}{\pi}\frac{\log m}{m}\,\tr S.
	\end{equation}
	The additional shift cancels with the additional exponential function which only exists for $\alpha=1$.
\end{proof}

In this lemma we notice that the quantity
\begin{equation}
	\hat r_m(S):=\la\exp\left(-\frac{w_\alpha(USU^\dagger)}{m}\right)\ra
\end{equation}
will become important because of $\hat r_m^m(S)=\hat p_m(S)$. It is essentially the characteristic function of the rescaled probability density $p_1$. Another quantity relating to the difference between those two characteristic function is
\begin{equation}\label{D0}
D:=\log \frac{\hat p_m(S)}{\hat p_\infty(S)}=m\log\hat r_m(S)+\la w_\alpha(USU^\dagger)\ra,
\end{equation}
which will be also exploited in the proofs.

\section{Regularities of the stable distributions}\label{sec:reg}

Recall that the support of a measure $h$, denoted by $\mathrm{supp}(h)$, is defined as the set of all points, for which any open neighbourhood of those points has a positive measure. 

In a first step to establish the regularity statement of Theorem~\ref{trace_traceless}, we give a general lemma saying that every stable random vector can be decomposed into a linear combination of a fixed deterministic vector and a random stable vector. The dimension of the latter relates to the support of the measure. Indeed, the situation in Theorem~\ref{trace_traceless} is then only a very special case of this case. Despite we could not find a proof of this lemma or any equivalent version of it, although we strongly believe that it exists in some literature, we will prove it here for completeness.

\begin{lemma}\label{multi_marginal}
	Let $X=(X_1,\ldots,X_d)^\top\in\mathbb{R}^d$ be a $d$-dimensional stable random vector following $S(\alpha,\sigma h,\mu)$. If $\mathrm{span}(\mathrm{supp}(h))$ is a $p$-dimensional real vector space with $p\le d$, there exists an invertible linear transformation $\Lambda:\R^d\mapsto\R^d$, such that the last $(d-p)$ entries of the random vector $\Lambda X$ are fixed.
\end{lemma}
\begin{proof}
	Surely, one can find a set of basis vectors $E=\{e_1,\ldots,e_d\}$ such that $\mathrm{span}(\mathrm{supp}(h))=\mathrm{span}\{e_1,\ldots,e_p\}$ (we suppose they are orthonormal for simplicity). Then we construct $\Lambda^{-1}$ by mapping the canonical basis vector $(0,\ldots,1,\ldots,0)$, who has only the $j$-th entry being $1$ and other entries are $0$, to the basis vector $e_j$, i.e., $\Lambda^{-1}=(e_1,\ldots, e_d)\in{\rm O}(d)$ is orthogonal. Then the characteristic function of $\Lambda X$ reads
	\begin{equation}\label{lemma5.eq}
	\begin{split}
	\mathbb E\exp\left(i\xi^\top \Lambda X\right)&=\exp\left(-\sigma\int_{r\in\mathbb{S}^{d-1}}h(\dv r)\,\nu_\alpha\left(\xi^\top \Lambda r\right)+i\xi^\top\Lambda\mu\right)\\
	&=\exp\left(-\sigma\int_{\tilde r\in\mathbb S^{d-1}} h(\Lambda^{-1}\dv \tilde r)\,\nu_\alpha\left(\xi^\top {\tilde r}\right)+i\xi^\top\tilde\mu\right)
	\end{split}
	\end{equation}
	where we make the changes of variable ${\tilde r}=\Lambda r$ and ${\tilde \mu}=\Lambda \mu$. 
	
	In the next step we extend the spectral measure from the sphere $\mathbb{S}^{d-1}$ to the whole vector space $\mathbb{R}^d$ by the definition
	\begin{equation}
	\int_{\R^d}\tilde h(\dv \tilde r)\varphi(\tilde r):=\int_{\mathbb S^{d-1}}h(\Lambda^{-1}\dv\tilde r)\varphi(\tilde r)
	\end{equation} 
	for any bounded measurable function $\varphi$ on $\R^d$. In other words, the support of $\tilde h$ does not intersect with the complement of the unit sphere. In particular, it also holds true that $\mathrm{supp}(\tilde h)=\Lambda\,\mathrm{supp}(h)=\{(\tilde{r}_1,\ldots,\tilde{r}_d)^\top\in\R^d|\tilde r_{p+1}=\ldots=\tilde r_d=0\}=\mathrm{span}(\Lambda E)$ as a closed set on the unit sphere is also closed in $\R^d$. 
	
	Now we will show that $\tilde h(\dv \tilde r)$ is a product measure of $\tilde h'(\dv \tilde r_1\ldots\dv \tilde r_p)$, which is a probability measure on $\mathrm{span}(\Lambda E)$, and a Dirac measure $\delta(\dv \tilde r_{p+1}\ldots\dv \tilde r_d)$. Explicitly, it means
	\begin{equation}
		\tilde h(\dv \tilde r)=\tilde h'(\dv \tilde r_1\ldots\dv \tilde r_p)\,\delta(\dv \tilde r_{p+1}\ldots\dv \tilde r_d)
	\end{equation}
	and the Dirac measure is given by 
	\begin{equation}
	\int_{\mathbb{R}^{d-p}}\delta(\dv \tilde r_{p+1}\ldots\dv \tilde r_d) \vartheta(r_{p+1},\ldots,r_d)=\vartheta(0,\ldots,0)
	\end{equation}
	for any bounded measurable function $\vartheta$ on $\R^{d-p}$.
	We construct $\tilde h'$ as follows: for $A\subset\R^p$ a measurable set of $\mathrm{span}(\Lambda E)$, we define
	\begin{equation}
		\tilde h'(A):=\tilde h(A\times\mathrm{span}(\Lambda E)^\perp)
	\end{equation}
	where $\mathrm{span}(\Lambda E)^\perp\simeq\R^{d-p}$ is the orthogonal complement of $\mathrm{span}(\Lambda E)$. What we need to show  is that for all $A\in\mathrm{span}(\Lambda E)$ and $B\in\mathrm{span}(\Lambda E)^{\perp}$,
	\begin{equation}
		\tilde h(A\times B)=\tilde h'(A)\ind_B(0)
	\end{equation}
	with $\ind_B(0)$ the indicator function which is only non-vanishing when $0\in B$.
	Thus when $0\not\in B$, the left hand side is zero as $A\times B$ does not intersect $\mathrm{span}(\Lambda E)$, where the right hand side also vanishes. When $0\in B$, one has
	\begin{equation}
		h(A\times B)=h(A\times B)+h(A\times (\mathrm{span}(\Lambda E)^\perp\backslash B))=h(A\times \mathrm{span}(\Lambda E)^\perp)
	\end{equation}
	which gives us the claim. Furthermore, it is also clear that
	\begin{equation}
		\mathrm{supp}(\tilde h')\subset \mathbb{S}^{p-1}
	\end{equation}
	because $\Lambda$ is orthogonal and does not change the normalisation of vectors.
	
	Utilizing this claim, Eq.~\eqref{lemma5.eq} can be rewritten as
	\begin{equation}
		\mathbb E\exp\left(i\xi^\top \Lambda X\right)=\exp\left(-\sigma\int_{(\tilde r_1,\ldots,\tilde r_p)\in\,\mathbb{S}^{p-1}}\tilde h'(\dv \tilde r_1\ldots\dv\tilde r_p)\,\nu_\alpha\left(\sum_{j=1}^p\xi_j\tilde r_j\right)+i\xi\tilde\mu^\top\right)
	\end{equation}
	which is a product of the characteristic function of a stable distribution $S(\alpha,\sigma\tilde h',(\tilde\mu_1,\ldots,\tilde\mu_p))$ and the characteristic function of the fixed vector $(\tilde\mu_{p+1},\ldots,\tilde\mu_d)$. This finishes the proof.
\end{proof}

In the next step in proving Theorem~\ref{trace_traceless}, we need to relate the knowledge about the dimension of the span of the support of the spectral distribution and the existence of a bounded continuous probability density function for the corresponding stable distribution. This is done in the following lemma.

\begin{lemma}\label{regularity_multi_stable}
	For a $d$-dimensional multivariate stable distribution $S(\alpha,\sigma h,\mu)$, the following statements are equivalent:
	\begin{enumerate}
		\item $\mathrm{span}(\mathrm{supp}(h))= \R^d$;
		\item $S(\alpha,\sigma h,\mu)$ has a bounded continuous probability density function $f$ on $\R^d$.
	\end{enumerate}
\end{lemma}

The special case of this lemma where the spectral measure is symmetric also appears in discussions in~\cite{DN02}. We present a proof here only for the sake of completeness.

\begin{proof}
	First we prove (1)$\Rightarrow $(2). Let $\xi$ be a unit vector in $\mathbb{R}^d$ and denote 
	\begin{equation}\label{def.mh}
		m_h:=\inf_{\|\xi\|^2=1}\int_{\mathbb S^{d-1}} h(\dv r)\left|r^\top\xi\right|^\alpha\ge 0
	\end{equation}
	which is essentially the lower bound of minus the real part of the logarithm of the characteristic function of the stable distribution $S(\alpha,\sigma h,\mu)$. $\|\xi\|^2$ is the Euclidean norm of the vector $\xi$.
	We will first show that $m_h$ is strictly positive as then the integrand is uniformly bounded by an integrable exponential function, see below. 
	
	Let us assume $m_h=0$. Then, this would imply that there is a sequence $\{\xi_n\}_{n\in\mathbb{N}}\subset\mathbb{S}^{d-1}$ so that
	\begin{equation}
		\int_{\mathbb S^{d-1}} h(\dv r)\left|r^\top\xi_n\right|^\alpha\leq\frac{1}{n}.
	\end{equation}
	As the sphere is bounded this sequence must have a convergent subsequence $\{\xi_{n_l}\}_{l\in\mathbb{N}}\subset\mathbb{S}^{d-1}$ due to the Bolzano-Weierstrass theorem, and the lack of a boundary of $\mathbb S^{d-1}$ says that the limiting point  $\xi_\infty$ lies also on the sphere. Thus it is
	\begin{equation}
		\int_{\mathbb S^{d-1}} h(\dv r)\left|r^\top\xi_\infty\right|^\alpha=0.
	\end{equation}
	This, however, means that the non-negative integrand must vanish on $\mathrm{supp}(h)$, i.e., $\left|r^\top\xi_\infty\right|^\alpha=0$, because the spectral measure $h$ is non-negative. More formalised, for the set $R_\varepsilon=\{r\in\mathbb{S}^{d-1}: |r^\top\xi_\infty|>\varepsilon\}$ with $\varepsilon>0$ the following holds true
	\begin{equation}
		0=\int_{\mathbb S^{d-1}} h(\dv r)\left|r^\top\xi_\infty\right|^\alpha=\int_{R_\varepsilon} h(\dv r)\left|r^\top\xi_\infty\right|^\alpha+\int_{\mathbb S^{d-1}\setminus R_\varepsilon} h(\dv r)\left|r^\top\xi_\infty\right|^\alpha\geq h(R_\varepsilon) \varepsilon^\alpha \geq0.
	\end{equation}
	As $\varepsilon>0$, it must be $h(R_\varepsilon)=0$ for all fixed $\varepsilon>0$. This can only be true when $\mathrm{supp}(h)$ is a subset of the orthogonal complement of $\xi_\infty$, which contradicts with the fact that $\mathrm{span}(\mathrm{supp}(h))=\R^d$.  Therefore, it must be $m_h>0$.

	Now, let $x\in\R^d$, $\zeta\in\mathbb{S}^{d-1}$ and $\varepsilon>0$ being small enough as we will take $\varepsilon\to0$. The existence and boundedness of the probability density function $S(\alpha,\sigma h,\mu)$ can be seen by the bound
	\begin{equation}
	\begin{split}
	\sup_{x\in\R^d}|S(\alpha,\sigma h,\mu;x)|\le&\sup_{x\in\R^d}\int_{\R^d}\dv \xi\, \left|\exp\left(-\sigma\int_{r\in\mathbb{S}^{d-1}}h(\dv r)\,\nu_\alpha\left(\xi^\top \Lambda r\right)+i\xi^\top\Lambda\mu-i\xi^\top x\right)\right|\\
	\le& \int_{\R^d}\dv\xi\,\exp(-\sigma m_h\|\xi\|^\alpha)<\infty.
	\end{split}
	\end{equation}
	For the continuity  we can utilise the bound
	\begin{equation}
	\begin{split}
		\int_{\R^d}\dv \xi\, \left|\exp\left(-\sigma\int_{r\in\mathbb{S}^{d-1}}h(\dv r)\,\nu_\alpha\left(\xi^\top \Lambda r\right)+i\xi^\top\Lambda\mu-i\xi^\top x\right)\left(e^{-i\varepsilon\xi^\top\zeta }-1\right)\right|&\le 2\int_{\R^d}\dv\xi\,\exp(-\sigma m_h\|\xi\|^\alpha)<\infty,
	\end{split}
	\end{equation}
	and the point-wise convergence of the integrand when $\varepsilon\to0$ to apply the dominated convergence theorem. This leads to
	\begin{equation}
	\begin{split}
	&\lim_{\varepsilon\mapsto 0}[S(\alpha,\sigma h,\mu;x)-S(\alpha,\sigma h,\mu;x+\varepsilon\zeta)]\\
		=&\lim_{\varepsilon\mapsto 0}\int_{\R^d}\dv \xi\, \exp\left(-\sigma\int_{r\in\mathbb{S}^{d-1}}h(\dv r)\,\nu_\alpha\left(\xi^\top \Lambda r\right)+i\xi^\top\Lambda\mu-i\xi^\top x\right)\left(e^{-i\xi^\top \varepsilon}-1\right)=0.
		\end{split}
	\end{equation}
	
	Finally, we prove the converse (2)$\Rightarrow$(1) by contradiction. Assuming that $\mathrm{span}(\mathrm{supp}(h))=V$ with $V$ being a $p$-dimensional non-trivial subspace of $\mathbb{R}^d$. Then by Lemma~\ref{multi_marginal}, there is an invertible linear transformation $\Lambda$ where the last $(d-p)$ entries of $\Lambda X$ with  the random vector $X$ drawn from $S(\alpha,\sigma h,\mu)$ is fixed in the orthogonal complement of $V$. We call this fixed $(d-p)$-dimensional vector $y_0$. Then the support of $h$ must have its last $(d-p)$ entries be exactly $y_0$, and so its interior (in the sense as an open set in $\R^d$) is empty. On the contrary, a probability density function of $\R^d$ must have some non-empty open set in the interior of its support, which constitutes a contradiction.
\end{proof}

Lemma~\ref{multi_marginal} and Lemma~\ref{regularity_multi_stable} indicate that one can always choose a convenient basis for any stable random vector, such that it becomes a Cartesian product of two independent vectors. The first one is then a stable random vector with a bounded continuous probability density function. The second one is deterministic and, thence, fixed. Applying the previous two lemmas to the matrix invariant stable distribution gives us a very short proof of Theorem~\ref{trace_traceless}.

\begin{proof}[Proof of Theorem~\ref{trace_traceless}]
	We notice that $D_X=\mathrm{span}(\mathrm{supp}(F))$ is a (linear) subspace of $\Herm(N)$ and it is invariant under $\U(N)$-conjugation. Therefore, there are only two non-trivial invariant subspaces of $\Herm(N)$ under $\U(N)$-conjugation which are the subspace of the trace times the identity matrix $I_N$ and the one of the traceless Hermitian matrices $\Herm_0(N)$, see e.g.,~\cite{Hallbook}. This implies that only one of four different cases of $D_X$ is possible: $D_X=\{0\}$, $D_X=\{xI_N:x\in\mathbb{R}\}$, $D_X=\Herm_0(N)$, or $D_X=\Herm(N)$. 
	
	The first case $D_X=\{0\}$ can be excluded as $X$ should be a non-zero random matrix, meaning there has to be a point of the support of $F$ that must be not equal to the zero matrix. The second case tells us the density of $X$ concentrates on its trace, and so $X$ is only a real univariate stable random variable times the identity matrix $I_N$. Respectively, the third and last cases correspond to a traceless Hermitian random matrix $X\in\Herm_0(N)$ and an $X\in\Herm(N)$ with no restrictions on it. In these two cases, by Lemma~\ref{multi_marginal} and Lemma~\ref{regularity_multi_stable} we have the desired regularity properties.
\end{proof}

With the help of Lemma~\ref{regularity_multi_stable} we can also prove the first statement of Theorem~\ref{thm2} about the equivalence of the regularity between $p_1$ and $p_\infty$

\begin{proof}[Proof of Theorem~\ref{thm2} statement (1)]
	First, we show that the support of $H$ spans $W$ if and only if the support of $H_\eig$ spans $\R^N$ or $\R^N_0=\{(r_1,\ldots,r_N)^\top\mathbb{R}^N:\sum_{j=1}^Nr_j=0\}$ respectively. Certainly, this follows from the unitary invariance of $H$ as it is 
	\begin{equation}\label{H-Heig-rel}
	H(\dv R)=H_\eig(\dv r)\mu_{\rm Haar}(\dv U)
	\end{equation}
	where we have employed the eigenvalue decomposition $R=U^\dagger\diag(r)U$ with $r\in\mathbb{R}$, $U\in\U(N)/[\U(1)]^N$, and $\mu_{\rm Haar}(\dv U)$ is the Haar measure on the coset $\U(N)/[\U(1)]^N$. In particular $U$ and $\diag(r)$ are two independent random matrices. It has to be the coset $\U(N)/[\U(1)]^N$ and not $\U(N)$ as the subgroup $[\U(1)]^N$ are the diagonal unitary matrices which commute with every diagonal matrix. Thus, it is $R=U^\dagger\diag(r)U=(VU)^\dagger\diag(r)VU$ for any $V\in[\U(1)]^N$.
	
	Summarising, it is $\mathrm{supp}(H)=\{U^\dagger\diag(r)U:U\in\U(N)/[\U(1)]^N, r\in\mathrm{supp}(h_\eig)\}$. From this we can read off that when $\mathrm{span}[\mathrm{supp}(h_\eig)]=\R^N$ it is $\mathrm{supp}(H)=\Herm(N)$ and when $\mathrm{span}[\mathrm{supp}(h_\eig)]=\R^N_0$ it is $\mathrm{supp}(H)=\Herm_0(N)$.
	
	Assuming the converse $\mathrm{span}(\mathrm{supp}(H))=W$, we can construct a contradiction when $\mathrm{span}[\mathrm{supp}(h_\eig)]\neq\R^N$ or $\mathrm{span}[\mathrm{supp}(h_\eig)]\neq\R^N_0$, respectively. The dimension of $\mathrm{span}(\mathrm{supp}(H))$ is $N^2$ for $W=\Herm(N)$ and $N^2-1$ for  $W=\Herm_0(N)$. The dimension of the set $W=\mathrm{span}\{U^\dagger\diag(r)U:U\in\U(N)/[\U(1)]^N, r\in\mathrm{supp}(h_\eig)\}$ will be also equal to or less than $\dim(\U(N)/[\U(1)]^N)+\dim(\mathrm{span}[\mathrm{supp}(h_\eig)])=N(N-1)+\dim(\mathrm{span}[\mathrm{supp}(h_\eig)])$. When $\mathrm{span}[\mathrm{supp}(h_\eig)]\neq\mathbb{R}^N$ it must be $\dim(\mathrm{span}[\mathrm{supp}(h_\eig)])<N$ so that we get a contradiction for $W=\Herm(N)$ as the dimensions do not match. Similarly for $W=\Herm_0(N)$ and $\mathrm{span}[\mathrm{supp}(h_\eig)]\neq\mathbb{R}_0^N$, it is $\dim(\mathrm{span}[\mathrm{supp}(h_\eig)])<N-1$ which leads to a contradiction. We note that in the latter case it cannot be $\dim(\mathrm{span}[\mathrm{supp}(h_\eig)])=N$ because the condition of Hermitian matrices being traceless carries over to its eigenvalues, i.e., $0=\tr R=\sum_{j=1}^Nr_j$. Hence, $\mathrm{span}[\mathrm{supp}(h_\eig)]$ must be a proper subspace of $\mathbb{R}^N$.
	
	By Lemma~\ref{regularity_multi_stable}, if $p_1$ is a bounded continuous probability density function on $W$ then so is $p_\infty$, and vice versa. Also it is clear that when $p_1$ is bounded and continuous, any of its i.i.d. sum (with any fixed shift) will be also bounded and continuous. For example, the boundedness of $p_m$ can be seen by
	\begin{equation}
	\begin{split}
		\sup_{X\in W}|p_m(X)|\leq\int_W |\hat p_m(S)|\dv S&=\int_W |\hat r_m(S)|^m\dv S \\
		&\le \sup_{S\in W}|\hat r_m(S)|^{m-1}\dv S\int_W |\hat r_m(S)|\dv S\\
		&= \sup_{S\in W}|\hat p_1(m^{1/\alpha}S)|^{m-1}\int_W |\hat p_1(m^{1/\alpha}S)|\dv S<\infty.
	\end{split}
	\end{equation}
	The argument of continuity works in a similar way.
\end{proof}

\section{Upper bound of the rate of convergence}\label{s.proof}

In order to prove the second part of Theorem~\ref{thm2}, we work on the level of an integral against the characteristic functions $\hat{p}_m$ and $\hat{p}_\infty$, i.e,
	\begin{equation}\label{eq3.1}
	\begin{split}
		|p_m(X)-p_\infty(X)|&=\left|\int_W \frac{\dv S}{c_W}\left(\la\exp\left[-\frac{1}{m}w_\alpha(USU^\dagger)\right]\ra^m-\exp\left[-\la w_\alpha(USU^\dagger]\ra\right)\right)\exp(-i\tr XS)\right|,
	\end{split}
	\end{equation}
	where $c_W$  is the proper normalisation of the inverse Fourier transform which only depends on the vector space $W$.
	
 The integral~\eqref{eq3.1} involves exponential functions comprising the function $w_\alpha$ in the exponent. It guarantees the convergence of the exponents, and its properties, especially its bounds of the real parts, are crucial for deriving the claimed bound~\eqref{thm.bound}. Hence we will start the proof by discussing the function $w_\alpha$ in subsection~\ref{sec:w-alpha}.
 
 The bounds of $w_\alpha$ obtained in subsection~\ref{sec:w-alpha} will be needed to bound each of the three terms of
	\begin{equation}\label{eq3.2}
	\begin{split}
		|p_m(X)-p_\infty(X)|\le&\int_W \frac{\dv S}{c_W}\left|\la\exp\left[-\frac{1}{m}w_\alpha(USU^\dagger)\right]\ra^m-\exp\left(-\la w_\alpha(USU^\dagger)\ra\right)\right|\\
		\le& \overbrace{\int_{\tr S^2\le T^2}\frac{\dv S}{c_W}\left|\la\exp\left[-\frac{1}{m}w_\alpha(USU^\dagger)\right]\ra^m-\exp\left(-\la w_\alpha(USU^\dagger)\ra\right)\right|}^{=:I_1}+\\
		&\underbrace{\int_{\tr S^2>T^2}\frac{\dv S}{c_W} \left|\la\exp\left[-\frac{1}{m}w_\alpha(USU^\dagger)\right]\ra^m\right|}_{=:I_2} +\underbrace{\int_{\tr S^2>T^2}\frac{\dv S}{c_W} \left|\exp\left[-\la w_\alpha(USU^\dagger)\ra\right]\right|}_{=:I_3}.
	\end{split}
	\end{equation}
This is a natural splitting into the main contribution of the integral within the hyper-ball given by $\tr S^2\le T^2$ for a suitable $T>0$ which will be fixed in the end and integrals over the tails. The tail behaviour of  $\hat{p}_m$ and $\hat{p}_\infty$, especially the integral over their moduli is performed in subsection~\ref{sec:tail} and the one of the difference in the bulk restricted by  $\tr S^2\le T^2$ is shown in subsection~\ref{sec:bulk}. Those pieces are put together to a proof of the second statement of Theorem~\ref{thm2} in subsection~\ref{sec:proof.thm2.2}.

\subsection{Properties of the function $w_\alpha$}\label{sec:w-alpha}

We would like to recall the definition~\eqref{def.w-alpha} of the function $w_\alpha$. From those we can read off the following immediate properties.

\begin{lemma}\label{walpha}
	One has the following properties of $w_\alpha$:
	\begin{enumerate}
		\item $w_\alpha$ is continuous in $X\in\Herm(N)$;
		\item the real part of $w_\alpha$ is
		\begin{equation}
			\re w_\alpha(X)=\int_{\mathbb S^{N-1}}H_\eig(\dv r)|\tr \diag(r)X|^\alpha\ge    0;
		\end{equation}
		\item $w_\alpha$ has the following absolute bound by a monotonically increasing function in the variable $\tr X^2$,
		\begin{equation}\label{kappa}
			|w_\alpha(X)|\le
			\begin{cases}\displaystyle
			\frac{(\tr X^2)^{\alpha/2}}{|\cos(\pi\alpha/2)|},&\alpha\in(0,1)\cup(1,2],\\\\\displaystyle
			\sqrt{\tr X^2}\sqrt{1+\left(\frac{\log(\tr X^2)}{\pi}\right)^2}
			,&\alpha= 1,
			\end{cases}
		\end{equation}
	which is also a piecewise continuous function in $\alpha$.
	\end{enumerate}
\end{lemma}
\begin{proof}
	As both statement (1) and (2) are surely evident when considering the definitions~\eqref{nualpha} and~\eqref{def.w-alpha}, we turn to the proof of the third statement. For $\alpha\ne 1$ we have
	\begin{equation}
	\left|w_\alpha(X)\right| \le \int_{\mathbb S^{N-1}}H_\eig(\dv r) \left|\frac{(-i\tr \diag(r)X)^\alpha}{\cos(\pi\alpha/2)}\right| \le \frac{(\tr X^2)^{\alpha/2}}{|\cos(\pi\alpha/2)|}
	\end{equation}
	for any $U\in\U(N)$, where we have made use of the fact that $|\tr \diag(r)X|\le (\tr X^2)^{1/2}$ for $\tr \diag(r)^2=1$. 
	
	In the $\alpha=1$ case, we exploit that the function $x[1+(\log(x)/\pi)^2]$ is strictly increasing for $x>0$ because of its first derivative
	\begin{equation}
	\frac{d}{dx}\left(x\left[1+\left(\frac{\log(x)}{\pi}\right)^2\right]\right)=1+\left(\frac{\log(x)}{\pi}\right)^2+\frac{2\log(x)}{\pi^2}=1-\frac{1}{\pi^2}+\frac{[\log(x)+1]^2}{\pi^2}>0.
	\end{equation}
	The latter can be exploited because of
	\begin{equation}
	\left|\,|x|+\frac{2i}{\pi}x\log x\right|=\sqrt{x^2\left[1+\left(\frac{\log(x^2)}{\pi}\right)^2\right]}
	\end{equation}
	where we plug in $x=\tr \diag(r)X$. Once we make use of $|\tr \diag(r)X|\le (\tr X^2)^{1/2}$ we find the desired bound.
\end{proof}

The following lemma gives us three equivalent statements on whether the support of the spectral measure (equivalently, the uniformly continuous part of the stable distribution) spans a subspace $W$. 

\begin{lemma}\label{lem:walpha-bounds}
	Let $W$ be an invariant subspace of $\Herm(N)$, $\mathbb S_W$ be the intersection of $W$ with the unit sphere $\mathbb S_{\Herm(N)}$ (meaning it is also a sphere but in the subspace $W$), and we define
		\begin{equation}
		m_H:=\inf_{S\in \mathbb S_W}\la \re w_\alpha(USU^\dagger)\ra.
		\end{equation}
		 The following statements are equivalent:
	\begin{enumerate}
		\item the support of $H$ spans $W$;
		\item it is $m_H>0$;
		\item for any $S\in \mathbb S_W$, there exists a non-empty open set $\Omega_S^\circ\subset \U(N)$ where for all $U\in\Omega_S^\circ$ it holds
		\begin{equation}
		\re w_\alpha(USU^\dagger)\ge m_H >0;
		\end{equation}
		\item there exists a positive constant  $c_H\in(0,1)$ such that for any $\gamma>0$ we can bound
		\begin{equation}
		\sup_{S\in\mathbb{S}_W}\la\exp\left(-\gamma\re w_\alpha(USU^\dagger) \right)\ra= 1-(1-e^{-\gamma m_H})c_H<1.   
		\end{equation}
	\end{enumerate}
\end{lemma}
\begin{proof}
	First we prove (1)$\Rightarrow$ (2), where we recall that 
		\begin{equation}\label{w-full}
		\la \re w_\alpha(USU^\dagger)\ra=\int_{\mathbb{S}_W} H(\dv R)|\,\tr RS|^\alpha=\int_{\mathbb{S}_{W}} H(\dv R)|\,\tr RS|^\alpha
		\end{equation}
		due to the factorisation~\eqref{H-Heig-rel} of the unitarily invariant measure $H$. With this in mind, we make use of the quantity $m_H$ defined in~\eqref{def.mh} where we have replaced $h\to H$, $\mathbb{R}^d\to W$  and $\mathbb{S}^{d-1}\to\mathbb{S}_W$. In the proof of Lemma~\ref{regularity_multi_stable} we have shown that this infinum is strictly positive so that it is
		\begin{equation}
		\la \re w_\alpha(USU^\dagger)\ra\geq\inf_{\tr \tilde S^2=1}\int_{\mathbb{S}_{W}} H(\dv R)|\,\tr R\tilde S|^\alpha=m_H>0,
		\end{equation}
		which has been the aim to show. The constant is also equal to $\inf_{S\in \mathbb S_W}\re w_\alpha(S)$.
		
	The converse direction (2)$\Rightarrow$ (1)  is also based on~\eqref{w-full}. Let us assume that $H$ only spans a proper subspace $\tilde{W}\subset W$. Then there is a unit vector $S_0\in \mathbb{S}_W$ which is orthogonal to $\tilde{W}$. But this means
		\begin{equation}
		\la \re w_\alpha(US_0U^\dagger)\ra=\int_{\mathbb{S}_W} H(\dv R)|\,\tr RS_0|^\alpha=\int_{\mathbb{S}_{\tilde{W}}} H(\dv R)|\,\tr RS_0|^\alpha=0,
		\end{equation}
		which is in contradiction to our starting property $\inf_{S\in \mathbb S_W}\la \re w_\alpha(USU^\dagger)\ra  >0$.

	Next we prove (2)$\Rightarrow$(3) by contradiction. For this purpose, we split the unitary group into two disjoint subsets $\U(N)=\Omega_S\cup\Omega_S^c$, such that $\re w_\alpha(USU^\dagger)\ge m_H$ for any $U\in \Omega_S$ and $\re w_\alpha(USU^\dagger)<m_H$ for any $U\in \Omega_S^c$. The set $\Omega_S^c$ is open because the map $U\mapsto \re w_\alpha(USU^\dagger)$ is continuous as it is evident from the definitions~\eqref{nualpha} and~\eqref{def.w-alpha}. Hence, the set $\Omega_S^c$ and the complement $\Omega_S$ are measurable with respect to the normalised Haar measure $\mu_{\rm Haar}$ of $\U(N)$.
	
	 Let us assume that $\mu_{\rm Haar}(\Omega_S)=0$. As the normalised Haar measure is a probability measure, this assumption implies $\mu_{\mathrm{Haar}}(\Omega_S^c)=1$ and
	\begin{equation}
		\int_{\Omega_S}\mu_{\rm Haar}(\dv U)\,\re w_\alpha(USU^\dagger)=0.
	\end{equation}
	Then by statement (2), one has
	\begin{equation}
		m_H\le \la\re w_\alpha(USU^\dagger) \ra=\int_{\Omega_S^c}\mu_{\rm Haar}(\dv U)\,\re w_\alpha(USU^\dagger)<\int_{\Omega_S^c}\mu_{\rm Haar}(\dv U)\,m_H= m_H \mu_{\mathrm{Haar}}(\Omega_S^c)=m_H,
	\end{equation}
	which is a contradiction. Therefore, it must be $\mu_{\rm Haar}(\Omega_S)>0$ and
	\begin{equation}
	\int_{\Omega_S}\mu_{\rm Haar}(\dv U)\,\re w_\alpha(USU^\dagger)\geq m_H\mu_{\rm Haar}(\Omega_S)>0.
	\end{equation}
	The boundary of $\Omega_S$ is lower dimensional and, thence, must have Haar measure $0$. Therefore, the open set $\Omega_S^\circ=\Omega_S\setminus\partial\Omega_S$ must be non-empty, and every $U\in\Omega_S^\circ$ satisfies $\re w_\alpha(U'SU'^\dagger)\geq m_H>0$ with the very same constant as in statement (2), especially it is given by~\eqref{def.mh} as we have seen.
	
	To show (3)$\Rightarrow$(4), we assume the converse, namely that  $\sup_{S\in\mathbb{S}_W}\la\exp\left(-\gamma\re w_\alpha(USU^\dagger) \right)\ra=1$. This means that there is a sequence $\{S_n\}_{n\in\mathbb{N}}\subset\mathbb{S}_W$ with $1>\la\exp\left(-\gamma\re w_\alpha(US_nU^\dagger) \right)\ra>1-1/n$. The sequence must have a limit $S_\infty\in\mathbb{S}_W$ due Bolzano-Weierstrass because  $\mathbb{S}_W$ is compact, and this limit satisfies $\la\exp\left(-\gamma\re w_\alpha(US_\infty U^\dagger) \right)\ra=1$.
	
	Next, we split the integral into integrals over the disjoint subsets $\U(N)=\Omega_{S_\infty}^\circ\cup \left(\Omega_{S_\infty}^\circ\right)^c$, where we make use of the fact that $\Omega_{S_\infty}^\circ$ is non-empty and open and every $U\in\Omega_{S_\infty}^\circ$ satisfies $\re w_\alpha(U{S_\infty}U^\dagger)\geq m_H$. Being non-empty and open implies  $\mu_{\mathrm{Haar}}(\Omega_{S_\infty}^\circ)>0$, and we exploit
	\begin{equation}
		\exp\left(-\gamma\re w_\alpha(U{S_\infty}U^\dagger)\right)\le \exp(-\gamma m_H)<1
	\end{equation}
	for all $U\in\Omega_{S_\infty}^\circ$. This can be combined to the estimate
	\begin{equation}
	\begin{split}
		1=\la\exp\left(-\gamma w_\alpha(U{S_\infty}U^\dagger) \right)\ra&\le \left(\int_{\Omega_{S_\infty}^\circ}+\int_{\left(\Omega_{S_\infty}^\circ\right)^c}\right) \mu_{\rm Haar}(\dv U)\,\exp\left(-\gamma\re w_\alpha(U{S_\infty}U^\dagger)\right)\\
		&\le  \exp(-\gamma m_H)\,\mu_{\mathrm{Haar}}\,(\Omega_{S_\infty}^\circ)+\mu_{\mathrm{Haar}}(\left(\Omega_{S_\infty}^\circ\right)^c)<1
	\end{split}
	\end{equation}
	which is a contradiction.
	Let us recall that $\re w_\alpha(U{S_\infty}U^\dagger)\geq0$; see the second statement of Lemma~\ref{walpha}. The last inequality follows from the fact that $m_H>0$ is exactly positive. The formula for the bound can be found by identifying $c_H=\mu_{\mathrm{Haar}}\,(\Omega_{S_\infty}^\circ)\in(0,1)$.
	
	Finally we show  (4)$\Rightarrow$(2). This we do with a similar reasoning as in the proof of (2)$\Rightarrow$(3). Especially, for any fixed $S\in\mathbb{S}_W$ there is a non-empty open set $\tilde \Omega_S^\circ\subset \U(N)$ such that $\mu_{\rm Haar}(\tilde \Omega_S^\circ)>0$ and for all $U\in \tilde \Omega_S^\circ\subset \U(N)$ it holds true (for $\gamma=1$)
	\begin{equation}
		\exp\left(-\re w_\alpha(USU^\dagger) \right)\le 1-(1-e^{- m_H})c_H<1.
	\end{equation}
	Thus one has $\re w_\alpha(USU^\dagger)\geq-\log(1-(1-e^{- m_H})c_H)>0$ for all $U\in \tilde \Omega_S^\circ$. 
	
	Now we follow the ideas of the proof of (3)$\Rightarrow$(4), where we assume that the infinum $m_H$ is vanishing. Then, we find a convergent subsequece with a limiting point $\tilde{S}_\infty\in\mathbb{S}_W$ such that $\la \re w_\alpha(U\tilde{S}_\infty U^\dagger)\ra =0$. However this is in contradiction with
	\begin{equation}
	\la \re w_\alpha(U\tilde{S}_\infty U^\dagger)\ra \geq \int_{\Omega_{\tilde{S}_\infty}^\circ}\mu_{\rm Haar}(\dv U)\re w_\alpha(U\tilde{S}_\infty U^\dagger)\geq-\log(1-(1-e^{- m_H})c_H)\mu_{\rm Haar}(\Omega_{\tilde{S}_\infty}^\circ)>0.
	\end{equation}
	This concludes the proof.
\end{proof}

\subsection{Tail estimates}\label{sec:tail}

We recall that $W$ is an invariant subspace of $\Herm(N)$ being either $\Herm(N)$ or $\Herm_0(N)$, and $S\in W$. The Lebesgue measure $\dv S$ is assumed to be on $W$, induced by the one on $\Herm(N)$.

First, we prove a bound for the integral $I_3$ defined in~\eqref{eq3.2}.

\begin{lemma}\label{tail1}
	Let $m_H$ be the constant in the second statement of Lemma~\ref{lem:walpha-bounds}, $\mathrm{vol}(\mathbb S_W)$ the volume of the sphere $\mathbb S_W$, and $\Gamma$ is the Gamma function. The following bound holds true
	\begin{equation}
	\begin{split}
		I_3:=&\left|\int_{\tr S^2>T^2}\dv S\exp\left(-\la w_\alpha(USU^\dagger)\ra\right)\right|\\
		\le&\frac{2^{(\dim W)/\alpha-1}\mathrm{vol}(\mathbb S_W)}{\alpha\, m_H^{\dim W/\alpha}}\left[\Gamma\left(\frac{\dim W}{\alpha}\right)+(m_HT^\alpha)^{\dim W/\alpha-1}\right]e^{- m_HT^\alpha}
	\end{split}	
	\end{equation}
	for any $T>0$. 
\end{lemma}
\begin{proof}
	Employing a polar decomposition in $W\cong\mathbb{R}^{\dim(W)}$, namely $S=\rho \Theta$ with $\Theta\in\mathbb{S}_W$ and $\rho=(\tr S^2)^{1/2}$, one can readily compute
	\begin{equation}
	\begin{split}
	\left|\int_{\tr S^2>T^2}\dv S\exp\left(-\la w_\alpha(USU^\dagger)\ra\right)\right|\le& \int_{\tr S^2>T^2}\dv S\exp\left(-\la \re w_\alpha(USU^\dagger)\ra\right)\\
	\le &\mathrm{vol}(\mathbb S_W)\int_{T}^\infty\dv \rho\, \rho^{\dim W-1}\exp(- m_H \rho^\alpha),
	\end{split}
	\end{equation}
	where we made use of the second statement in Lemma~\ref{lem:walpha-bounds}. The integral is essentially an incomplete Gamma function,
	\begin{equation}
	\int_{T}^\infty\dv \rho\, \rho^{\dim W-1}\exp(- m_H \rho^\alpha)=\frac{\Gamma(\dim W/\alpha, m_HT^\alpha)}{\alpha\, m_H^{\dim W/\alpha}}.
	\end{equation}
	The incomplete Gamma function can be bounded from above as follows
	\begin{equation}
	\begin{split}
	\Gamma\left(\kappa,x\right)=&\int_{x}^\infty\dv t\, t^{\kappa-1}e^{-t}=e^{-x}\int_{0}^\infty\dv t\, (t+x)^{\kappa-1}e^{-t}\\
	\leq& e^{-x}\int_{0}^\infty\dv t\, ((2t)^{\kappa-1}+(2x)^{\kappa-1})e^{-t}=2^{\kappa-1}(\Gamma(\kappa)+x^{\kappa-1})e^{-x}
	\end{split}
	\end{equation}
	for any $x\geq0$ and $\kappa\geq1$. Indeed, $\kappa=\dim W/\alpha>1$ is given because $\alpha\in(0,2]$ and $\dim W\geq 3$ due to $N>1$, meaning the smallest space is $\Herm_0(2)$ which is three-dimensional. Collecting everything we find the claim.
\end{proof}

Next, we derive a bound for the integral $I_2$ defined in~\eqref{eq3.2}. For this purpose, we would like to point out that the exponential function $\exp\left(-\frac{1}{m}\re w_\alpha(S)\right)$ is also the characteristic function of a unitarily invariant stable ensemble, namely of
\begin{equation}
\begin{split}
S(\alpha, m^{-1}H_{\rm sym},0;X)=&\int_W\frac{dS}{c_W}\exp\left(-\frac{1}{m}\re w_\alpha(S)-i\tr SX\right)=m^{\dim W/\alpha}\,S(\alpha, H_{\rm sym},0;m^{1/\alpha} X)
\end{split}
\end{equation}
with $c_W$ the proper normalisation for the inverse Fourier transform on $W$. The spectral measure $H_{\rm sym}$ is the symmetrisation of $H$, which can be given in a distributional way as
\begin{equation}
\int_{\mathbb{S}_W} H_{\rm sym}(\dv R)\phi(R):=\int_{\mathbb{S}_W} H(\dv R)\frac{\phi(R)+\phi(-R)}{2}
\end{equation}
for any bounded test function $\phi$ on the sphere $\mathbb{S}_W$. We notice that for $X=0$ we obtain
\begin{equation}
\begin{split}
\int_W\frac{dS}{c_W}\exp\left(-\frac{1}{m}\re w_\alpha(S)\right)=m^{\dim W/\alpha}\,S(\alpha, H_{\rm sym},0;0)
\end{split}
\end{equation}
where $S(\alpha, H_{\rm sym},0;0)\in[0,\infty)$ is an $m$-independent finite constant because the distribution $S(\alpha, H_{\rm sym},0)$ is bounded and continuous on the vector space $W$ due to Theorem~\ref{trace_traceless}.

\begin{lemma}\label{tail2}
	Using the same notation as in Lemma~\ref{tail1}, we find for every $m\in\mathbb{N}$ the upper bound
	\begin{equation}
	I_2:=\int_{\tr S^2>T^2}\dv S\left|\la\exp\left(-\frac{1}{m}w_\alpha(USU^\dagger)\right)\ra^m\right|\le m^{\dim W/\alpha}\,S(\alpha, H_{\rm sym},0;0)\left(1-(1-e^{-T^\alpha m_H/m})c_H\right)^{m-1}
	\end{equation}
	for any $T>0$.
\end{lemma}
\begin{proof}
	We start with the upper bound
	\begin{equation}\label{eq.proof.tail2.a}
	\begin{split}
	&\int_{\tr S^2>T^2}\dv S\left|\la\exp\left(-\frac{1}{m}w_\alpha(USU^\dagger)\right)\ra^m\right|\le\int_{\tr S^2>T^2}\dv S\la\exp\left(-\frac{1}{m}\re w_\alpha(USU^\dagger)\right)\ra^m\\
	&\le \sup_{\tr S^2>T^2}\la\exp\left(-\frac{1}{m}\re w_\alpha(USU^\dagger)\right)\ra^{m-1}\int_{W}\dv S\la\exp\left(-\frac{1}{m}\re w_\alpha(USU^\dagger)\right)\ra.
	\end{split}
	\end{equation}
	The first factor can be bounded as follows
	\begin{equation}
	\begin{split}
	\sup_{\tr S^2>T^2}\la\exp\left(-\frac{1}{m}\re w_\alpha(USU^\dagger)\right)\ra^{m-1}=&\sup_{\tilde{s}>T}\sup_{S\in\mathbb{S}_W}\la\exp\left(-\frac{\tilde{s}^\alpha}{m}\re w_\alpha(USU^\dagger)\right)\ra^{m-1}\\
	=&\sup_{\tilde{s}>T}\left(1-(1-e^{-\tilde{s}^\alpha m_H/m})c_H\right)^{m-1}\\
	=&\left(1-(1-e^{-T^\alpha m_H/m})c_H\right)^{m-1}.
	\end{split}
	\end{equation}
	In the second line we have made use of the fourth statement of Lemma~\ref{lem:walpha-bounds} and in the third line we notice that $1-(1-e^{-\tilde{s}^\alpha m_H/m})c_H$ is strictly decreasing in $\tilde s>0$.
	
	The factor in~\eqref{eq.proof.tail2.a} can be simplified by interchanging the integral over $S$ with the one over $U$ as all integrals together are absolutely integrable. Moreover, we make use of the invariance $\dv (U^\dagger S U)=dS$ so that we can absorb the matrix $U$ in the $S$ integration. Hence, the average over $U$ is trivial and we arrive at
	\begin{equation}
	\int_{W}\dv S\la\exp\left(-\frac{1}{m}\re w_\alpha(USU^\dagger)\right)\ra=\int_{W}\dv S\exp\left(-\frac{1}{m}\re w_\alpha(S)\right)=m^{\dim W/\alpha}S(\alpha, H_{\rm sym},0;0).
	\end{equation}
	This finishes the proof.
\end{proof}

Now we are ready to combine the two bounds of $I_2$ and $I_3$ and simplify those to an exponential bound $T^\alpha$.

\begin{lemma}\label{tail_bound}
	Employing the notations of Lemmas~\ref{tail1} and~\ref{tail2}, there are two constants $d_H,D_H>0$, which only depend on the spectral measure $H$, such that for all $1<T^\alpha<m$, one has
	\begin{equation}\label{I2+I3}
	I_2+I_3\le D_H(1+m^{\dim W/\alpha})\exp[-d_H T^\alpha].
	\end{equation}
	In the case $T^\alpha=m^\varepsilon>1$ with $\varepsilon\in(0,1)$ this can be simplified to
	\begin{equation}\label{I2+I3.b}
	I_2+I_3\le \tilde{D}_H\exp[-\tilde{d}_H m^\varepsilon]
	\end{equation}
	with other two constants $\tilde{D}_H,\tilde{d}_H>0$.
\end{lemma}

The bound $T^\alpha<m$ can be actually extended to any $T^\alpha< k m$ for any $k>0$ without changing this bound. However, we do not need this bound as for proving the second statement of Theorem~\ref{thm2} we will choose a suitable $T$ satisfying $1\ll T^\alpha\ll m$ when $m\gg1$.

\begin{proof}
Let us start with the trivial observation that for any $a_1,a_2,c>0$ there is $b>1$ such that
\begin{equation}\label{inequ.gen}
1+a_1x^c\leq b\exp[a_2 x]\quad{\rm for\ all}\ x\geq0.
\end{equation}
This can be readily seen as the difference $1+a_1x^c- b\exp[a_2 x]$ is continuous in $x\geq0$ and the limit $x\to\infty$ diverges to $-\infty$ while for $x=0$ we have $1-b$. Hence, there must be a maximum of this difference in $x\geq0$ and for a suitably large $b>1$ this maximum has a negative value.

To obtain the first bound~\eqref{I2+I3}, we notice that for $I_3$ the bound in Lemma~\ref{tail1} is almost of the desired form. When pulling the Gamma function out of the bracket and exploiting~\eqref{inequ.gen}, we can approximate
\begin{equation}
\begin{split}
1+\frac{(m_HT^\alpha)^{\dim W/\alpha-1}}{\Gamma[\dim W/\alpha]}\leq&b_H\exp\left[\frac{m_HT^\alpha}{2}\right]
\end{split}
\end{equation}
for some suitably large $b_H>1$, because of $\dim W/\alpha>1$ and $m_HT^\alpha>0$. This means it is
\begin{equation}
I_3\leq  D_{H,3}\exp\left[-\frac{m_HT^\alpha}{2}\right]
\end{equation}
for all $T>0$ with a suitably large constant $D_{H,3}>0$.

For the bound of $I_2$ we focus on the term
\begin{equation}
\begin{split}
\left(1-(1-e^{-T^\alpha m_H/m})c_H\right)^{m-1}=&\left(1-(1-e^{-T^\alpha m_H/m})c_H\right)^{-1}\\
&\times\exp\left[T^\alpha m_H\frac{m}{T^\alpha m_H}\log\left(1-(1-e^{-T^\alpha m_H/m})c_H\right)\right].
\end{split}
\end{equation}
The function $1-(1-e^{-x})c_H$ is never zero for any $x\in[0,m_H]$ as it is strictly decreasing and takes its smallest value at $x=m_H$ with $1-(1-e^{-m_H})c_H>0$ because $c_H\in(0,1)$, see fourth statement of Lemma~\ref{lem:walpha-bounds}. Also the function $x^{-1}\log(1-(1-e^{-x})c_H)$ has a maximum as it is continuous on $x\in(0,1)$ and its limit $x\to0$ diverges to $-\infty$ while for $x\to 1$ it is $\log(1-(1-e^{-1})c_H)<0$. Moreover, it is $1-(1-e^{-1})c_H<1$ so that the value at the maximum must be negative as well. Denoting $d'_H=-m_H\max_{x\in(0,1]}\{x^{-1}\log(1-(1-e^{-x})c_H)\}$, it is
\begin{equation}
I_2\leq \frac{S(\alpha, H_{\rm sym},0;0)}{1-(1-e^{-m_H})c_H}m^{\dim W/\alpha}\exp[-d'_H T^\alpha].
\end{equation}
Choosing ${d}_H=\min\{d'_H, m_H/2\}$ and $D_H$ being the sum of the two constants in front of the exponential functions, we find the claimed bound~\eqref{I2+I3}.

For the bound~\eqref{I2+I3.b} we make use of the inequality~\eqref{inequ.gen} once again, where we identify $a_1=1$, $a_2=d_H/2$ and $c=\dim W/(\alpha\varepsilon)$. This closes the proof.
\end{proof}

\subsection{Bulk estimate}\label{sec:bulk}

Next, we turn to the more involved part of our estimates the integration $I_1$ over the bulk $\tr S^2<T^2$, see~\eqref{eq3.2}. To begin with, we recall the term $D$ from~\eqref{D0} and make an estimate for it.

\begin{lemma}[Estimate of the difference between characteristic functions]\label{lemma_difference}
	Recall that
	\begin{equation}\label{D}
		D:=m\log\la\exp\left(-\frac{w_\alpha(USU^\dagger)}{m}\right)\ra+\la w_\alpha(USU^\dagger)\ra;
	\end{equation}
	see Eq.~\eqref{D0}. Let $m>T>1$ and we assume
	\begin{equation}\label{bound.cond}
	\frac{e^{1/2}T^{\alpha}}{m|\cos(\pi\alpha/2)|}\leq\frac{1}{2}\ {\rm for}\ \alpha\in(0,1)\cup(1,2],\quad \frac{T}{m}\sqrt{1+\left(\frac{\log(T^2)}{\pi}\right)^2}\leq \frac{1}{2}\ {\rm for}\ \alpha=1.
	\end{equation}
	Then there is a constant $K>0$ independent of $T$ and $m$ (but it may dependent on $\alpha$) such that the difference can be uniformly bounded in $S\in\Herm(N)$ satisfying $\tr S^2\leq T^2$ as follows
	\begin{equation}
		|D|\le \frac{\left|\la w_\alpha(USU^\dagger)^2\ra-\la w_\alpha(USU^\dagger)\ra^2\right|}{2m}+K\varepsilon_{T,m},
	\end{equation}
where the error term is given by
\begin{equation}
	\varepsilon_{T,m}=\begin{cases}\displaystyle
		\frac{T^{3\alpha}}{m^2},&\alpha\in(0,1)\cup(1,2],\\\\
		\displaystyle
		\frac{[T\log(e\,T)]^3}{m^2}
		,&\alpha= 1.
	\end{cases}
\end{equation}
\end{lemma}

We note that conditions for these estimates can be readily satisfied when choosing $m$ large enough compared to $T>1$.

\begin{proof}
We will make use of the uniform bounds for the remainders of the Taylor expansion of the complex exponential function 
\begin{equation}
e^{z}=\sum_{j=0}^{l-1} \frac{z^j}{j!}+R_{e,l}(z)
\end{equation}
and for the complex logarithm
\begin{equation}
\log(1-z)=-\sum_{j=1}^{l-1} \frac{z^j}{j}+R_{\log,l}(z).
\end{equation}
When $|z|\leq 1/(2e^{1/2})<1/2$. Those estimates are
\begin{equation}\label{exp.Taylor}
|R_{e,l}(z)|=\left|e^{z}-\sum_{j=0}^{l-1} \frac{z^j}{j!}\right|\leq |z|^l\sum_{j=0}^\infty\frac{|z|^j}{(j+l)!}\leq |z|^l\sum_{j=0}^{\infty}\frac{(1/2)^j}{j!\,l!}\leq e^{1/2}\frac{|z|^l}{l!}
\end{equation}
and
\begin{equation}\label{log.Taylor}
|R_{\log,l}(z)|=\left|\log(1-z)+\sum_{j=1}^{l-1} \frac{z^j}{j}\right|\leq |z|^l\sum_{j=0}^\infty\frac{|z|^j}{j+l}\leq |z|^l\left(1+\sum_{j=1}^{\infty}\frac{(1/2)^j}{j}\right)\leq (1+\log2)|z|^l.
\end{equation}

Considering  the first term in~\eqref{D}, we notice that the quantity $w_\alpha(USU^\dagger)/m$  is bounded via~\eqref{kappa}. Specially for $\tr S^2<T^2$ we have
\begin{equation}\label{eq.diff.p1}
\left|\frac{w_\alpha(USU^\dagger)}{m}\right|\leq\begin{cases}\displaystyle
			\frac{T^{\alpha}}{m|\cos(\pi\alpha/2)|},&\alpha\in(0,1)\cup(1,2],\\\\\displaystyle
			\frac{T}{m}\sqrt{1+\left(\frac{\log(T^2)}{\pi}\right)^2}
			,&\alpha= 1,
			\end{cases}
\end{equation}
and by our assumptions~\eqref{bound.cond} it is smaller than $1/2$ in both cases. Thence, it is
\begin{equation}\label{eq.3.48}
	\begin{split}
\la\exp\left(-\frac{w_\alpha(USU^\dagger)}{m}\right)\ra&=1-\frac{\la w_\alpha(USU^\dagger)\ra}{m}+\frac{\la w_\alpha(USU^\dagger)^2\ra}{2m^2}+\la R_{e,3}\left(-\frac{w_\alpha(USU^\dagger)}{m}\right)\ra\\
&=1-\frac{\la w_\alpha(USU^\dagger)\ra}{m}+\la R_{e,2}\left(-\frac{w_\alpha(USU^\dagger)}{m}\right)\ra\\
&=1+\la R_{e,1}\left(-\frac{w_\alpha(USU^\dagger)}{m}\right)\ra
\end{split}
\end{equation}
with
\begin{equation}\label{eq.diff.p2}
\left|\la R_{e,l}\left(-\frac{w_\alpha(USU^\dagger)}{m}\right)\ra\right|\leq\begin{cases}\displaystyle
			\frac{e^{1/2}}{l!}\left[\frac{T^{\alpha}}{m|\cos(\pi\alpha/2)|}\right]^l,&\alpha\in(0,1)\cup(1,2],\\\\
			\displaystyle
			\frac{e^{1/2}}{l!}\left[\frac{T}{m}\sqrt{1+\left(\frac{\log(T^2)}{\pi}\right)^2}\right]^l
			,&\alpha= 1,
			\end{cases}
\end{equation}
for any $l\in\mathbb{N}$.

Next we take the logarithm and make use of its Taylor expansion~\eqref{log.Taylor} and the expansions given in~\eqref{eq.3.48}, which leads to
\begin{equation}
\begin{split}
m\log\la\exp\left(-\frac{w_\alpha(USU^\dagger)}{m}\right)\ra
=&-m\left(1-\la\exp\left(-\frac{w_\alpha(USU^\dagger)}{m}\right)\ra\right)-\frac{m}{2}\left(1-\la\exp\left(-\frac{w_\alpha(USU^\dagger)}{m}\right)\ra\right)^2\\&+m\, R_{\log,3}\left(1-\la\exp\left(-\frac{w_\alpha(USU^\dagger)}{m}\right)\ra\right)
\\=&-\la w_\alpha(USU^\dagger)\ra+\frac{\la w_\alpha(USU^\dagger)^2\ra}{2m}+m\la R_{e,3}\left(-\frac{w_\alpha(USU^\dagger)}{m}\right)\ra\\
&-\frac{m}{2}\left[\frac{\la w_\alpha(USU^\dagger)\ra}{m}-\la R_{e,2}\left(-\frac{w_\alpha(USU^\dagger)}{m}\right)\ra\right]^2\\
&+m\, R_{\log,3}\left(-\la R_{e,1}\left(-\frac{w_\alpha(USU^\dagger)}{m}\right)\ra\right).
\end{split}
\end{equation}
This means the modulus of the difference is bounded from above by
\begin{equation}
\begin{split}
|D|\leq&\frac{\left|\la w_\alpha(USU^\dagger)^2\ra-\la w_\alpha(USU^\dagger)\ra^2\right|}{2m}+m\left|\la R_{e,3}\left(-\frac{w_\alpha(USU^\dagger)}{m}\right)\ra\right|\\
&+\left|\la w_\alpha(USU^\dagger)\ra\la R_{e,2}\left(-\frac{w_\alpha(USU^\dagger)}{m}\right)\ra\right|+\frac{m}{2}\left|\la R_{e,2}\left(-\frac{w_\alpha(USU^\dagger)}{m}\right)\ra\right|^2\\
&+(1+\log2)m\left|\la R_{e,1}\left(-\frac{w_\alpha(USU^\dagger)}{m}\right)\ra\right|^3
\end{split}
\end{equation}
due to
\begin{equation}
\left|\la R_{e,1}\left(-\frac{w_\alpha(USU^\dagger)}{m}\right)\ra\right|\leq e^{1/2}\la\left|-\frac{w_\alpha(USU^\dagger)}{m}\right|\ra\leq \frac{1}{2}
\end{equation}
by our assumptions. The sum $s_3$ of three of the four correction terms
\begin{equation}
	\begin{split}
s_3:=&m\left|\la R_{e,3}\left(-\frac{w_\alpha(USU^\dagger)}{m}\right)\ra\right|+\left|\la w_\alpha(USU^\dagger)\ra\la R_{e,2}\left(-\frac{w_\alpha(USU^\dagger)}{m}\right)\ra\right|\\&+\left(1+\log2\right)m\left|\la R_{e,1}\left(-\frac{w_\alpha(USU^\dagger)}{m}\right)\ra\right|^3,
\end{split}
\end{equation}
 can be bounded by
\begin{equation}
s_3\leq\widetilde{K}_1\begin{cases}\displaystyle
			\frac{T^{3\alpha}}{m^2},&\alpha\in(0,1)\cup(1,2],\\\\
			\displaystyle
			\frac{T^3}{m^2}\left[1+\left(\frac{\log(T^2)}{\pi}\right)^2\right]^{3/2}
			,&\alpha= 1,
			\end{cases}
\end{equation}
for a suitable $\widetilde{K}_1>0$ because of~\eqref{eq.diff.p1} and~\eqref{eq.diff.p2}. The remaining fourth term has a similar bound,
\begin{equation}
\frac{m}{2}\left|\la R_{e,2}\left(-\frac{w_\alpha(USU^\dagger)}{m}\right)\ra\right|^2\leq\begin{cases}\displaystyle
			\frac{e\, T^{4\alpha}}{8m^3|\cos(\pi\alpha/2)|^4}\leq \frac{e^{1/2}\, T^{3\alpha}}{16m^2|\cos(\pi\alpha/2)|^3},&\hspace*{-1cm}\alpha\in(0,1)\cup(1,2],\\\\
			\displaystyle
			\frac{e\, T^4}{8m^3}\left[1+\left(\frac{\log(T^2)}{\pi}\right)^2\right]^{2}\leq\frac{e^{1/2}\, T^3}{16m^2}\left[1+\left(\frac{\log(T^2)}{\pi}\right)^2\right]^{3/2}
			,&\alpha= 1.
			\end{cases}
\end{equation}
	By noticing that  there is a $T$-independent constant  $\widetilde{K}_2>0$ such that $1+\left(\log(T^2)/\pi\right)^2\leq \widetilde{K}_2\log^2(e\,T)$ for $T>1$, we finish the proof also for $\alpha=1$.
\end{proof}

\subsection{Proof of Theorem~\ref{thm2} statement (2)}\label{sec:proof.thm2.2}

\begin{proof} 
Our strategy is to show that the leading order term in $m$ of the difference $|p_m(X)-p_\infty(X)|$ is given uniformly by $1/m$ times a constant. Yet, the remainder must be estimated as well such that it cannot contribute when taking the supremum. For this purpose, we start with~\eqref{eq3.2}, meaning
	\begin{equation}
	\begin{split}
		|p_m(X)-p_\infty(X)|=I_1+I_2+I_3
	\end{split}
	\end{equation}
	where $I_1, I_2, I_3$ are integrals depending on a truncation parameter $T$. A combined estimate of the second and third term is given in~\eqref{I2+I3.b} which gives an exponential suppression in $T$. So it remains to bound the first term. 
	
	Thence, we need to find a suitable $T$ such that $I_1$ can be bounded as claimed. For each $\alpha\in(0,2]$, we  choose an arbitrary $\varepsilon\in(0,\min(\alpha,1/4))$ and set $T^\alpha=m^\varepsilon>1$. One can check that in order to have the requirements~\eqref{bound.cond} for $T^\alpha=m^\varepsilon>2$ it is sufficient to have
	\begin{equation}\label{3.57}
		m\ge \begin{cases}\displaystyle\left(\frac{2e^{1/2}}{|\cos(\pi\alpha/2)|}\right)^{1/(1-\varepsilon)},&\alpha\in(0,1)\cup(1,2],\\\\\displaystyle
			3,&\alpha=1.
			\end{cases}
	\end{equation}
	The inequality for $\alpha=1$, follows from the fact that the function
	\begin{equation}
	q_\varepsilon(m)=\left[m^{\varepsilon-1}\sqrt{1+\frac{4\varepsilon^2}{\pi^2}\log^2(m)}\right]^2=m^{2\varepsilon-2}\left(1+\frac{4\varepsilon^2}{\pi^2}\log^2(m)\right)
	\end{equation}
	is strictly monotonously decreasing in $\log(m)$ for any $0<\varepsilon<1/4$ and $m\geq 1$, and $q_\varepsilon(3)$ is strictly monotonously increasing in $\varepsilon$ with a maximal value $q_{1/4}(3)\approx 0.19<1/4$. We cannot go down to $m=1$ or $m=2$ since $q_{1/4}(1),q_{1/4}(2)>1/4$ which would violate the requirements of Lemma~\ref{lemma_difference}.
	
	In conclusion, we define the $\varepsilon$-independent bound
	\begin{equation}
		m^{(\alpha)}=\begin{cases}
			\displaystyle\frac{2e^{1/2}}{|\cos(\pi\alpha/2)|},&\alpha\in(0,1)\cup(1,2],\\\\\displaystyle
			3,&\alpha=1,
			\end{cases}
	\end{equation}
	and choose $m\ge m^{(\alpha)}$, such that~\eqref{3.57} and, hence, \eqref{bound.cond} are satisfied (note that $\varepsilon<1/4$). In addition, Lemma~\ref{lemma_difference} also requires that $m>T$, which has already been satisfied because $\varepsilon<\alpha$.
	
	Moreover, $p_j$ (for any fixed $j$) and $p_\infty$ are bounded so that there is a constant $C>0$ depending only on $\alpha$ with
	\begin{equation}
		\sup_{X\in W}|p_j(X)-p_\infty(X)|\le\frac{C}{j}\text{ for all }1<j\le m^{(\alpha)}.
	\end{equation}
	Therefore, it is sufficient to concentrate only on the case $m\ge m^{(\alpha)}$. In particular, this consideration allows us to relax the bound to $m\geq 1$ as written in Theorem~\ref{thm2} statement (2) since we only exclude a finite number of possible values of $m$ when choosing $m> m^{(\alpha)}$. This exclusion may only change the constant $C$ in~\eqref{thm.bound} but the  $1/m$-bound remains the same.

	We can exploit the inequality $|e^z-1|\leq|z|e^{|z|}$ for $z\in\R$ for the difference $D$ of~\eqref{D} yielding
	\begin{equation}\label{3.61}
	I_1\le\int_{\tr S^2\le T^2}\frac{\dv S}{c_W}|D|e^{|D|}\exp\left(-\la \re w_\alpha(USU^\dagger)\ra\right).
	\end{equation}	
	We will bound the integral on the right hand side. The first inequality we make use of is
	\begin{equation}\label{3.62}
		\la \re w_\alpha(USU^\dagger)\ra\geq m_H(\tr S^2)^{\alpha/2}
	\end{equation}
	which is due to Lemma~\ref{lem:walpha-bounds} statement~(3). This gives a bound for the term $\exp\left(-\la \re w_\alpha(USU^\dagger)\ra\right)$. 
	
	Next we want to bound the term $|D|$ for which we need, see~\eqref{kappa},
	\begin{equation}\label{3.63}
		\frac{\left|\la w_\alpha(USU^\dagger)^2\ra-\la w_\alpha(USU^\dagger)\ra^2\right|}{2m}\leq \begin{cases}\displaystyle
			\frac{(\tr S^2)^{\alpha/2}}{m|\cos(\pi\alpha/2)|} ,&\alpha\in(0,1)\cup(1,2],\\\\
			\displaystyle
			\frac{1}{m}\sqrt{\tr S^2}\sqrt{1+\left(\frac{\log(\tr S^2)}{\pi}\right)^2}
			,&\alpha= 1.
		\end{cases}	
	\end{equation}
	Thus it is
		\begin{equation}\label{3.64}
		\begin{split}
			|D|\le&\frac{\left|\la w_\alpha(USU^\dagger)^2\ra-\la w_\alpha(USU^\dagger)\ra^2\right|}{2m}+K_2\begin{cases}\displaystyle
				\frac{T^{3\alpha}}{m^2},&\alpha\in(0,1)\cup(1,2],\\\\
				\displaystyle
				\frac{[T\log(T)]^3}{m^2}
				,&\alpha= 1,
			\end{cases}\\
			\le&\begin{cases}\displaystyle
				\frac{(\tr S^2)^{\alpha/2}}{m|\cos(\pi\alpha/2)|}+\frac{K_2}{m^{2-3\varepsilon}} ,&\alpha\in(0,1)\cup(1,2],\\\\
				\displaystyle
				\frac{1}{m}\sqrt{\tr S^2}\sqrt{1+\left(\frac{\log(\tr S^2)}{\pi}\right)^2}+\frac{K_2\,\varepsilon ^3\log^3 m}{m^{2-\varepsilon}}
				,&\alpha= 1,
			\end{cases}	
		\end{split}
	\end{equation}
	where the first inequality is due to Lemma~\ref{lemma_difference}. The second one is due to a substitution of~\eqref{3.63} into the first term and exploiting our choice $T^\alpha=m^{\varepsilon}$ for the second term. Recall that the requirements of Lemma~\ref{lemma_difference} have been resolved by our choice of $m\geq m^{(\alpha)}$ and $\varepsilon\in(0,\min(\alpha,1/4))$.
	
	The right hand side of the inequality~\eqref{3.64} can be further simplified. Starting from~\eqref{3.63} in combination with the inequality $\tr S^2\le T^2$ (domain of the corresponding integral) and $T^\alpha=m^\varepsilon$, we have
	\begin{equation}\label{3.63a}
		\frac{\left|\la w_\alpha(USU^\dagger)^2\ra-\la w_\alpha(USU^\dagger)\ra^2\right|}{2m}\leq \begin{cases}\displaystyle
			\frac{m^{\varepsilon-1}}{|\cos(\pi\alpha/2)|} ,&\alpha\in(0,1)\cup(1,2],\\\\
			\displaystyle
			 \sqrt{\tilde K_2}\,m^{\varepsilon-1}\log(e\, m^\varepsilon)
			,&\alpha= 1.
		\end{cases}	
	\end{equation}
 	In the $\alpha=1$ case we once again made use of $1+(\log(T^2)/\pi)^2\le\tilde K_2\log^2(e\, T)$. Substituting this into~\eqref{3.64}, one can see that $|D|$ is bounded by a decreasing function in the variable $m$. Thus when $m\ge 1$, one can conclude that for all $\alpha\in(0,2]$, there exists a constant $K_1$ independent of $m$ such that
 	\begin{equation}\label{3.66}
 		|D|\le K_1.
 	\end{equation}
	
	Finally, we substitute~\eqref{3.62},~\eqref{3.64} and~\eqref{3.66} into~\eqref{3.61} and find in the case $\alpha\in(0,1)\cup(1,2]$
	\begin{equation}
	\begin{split}
	I_1\le& \int_{\tr S^2<T^2}\frac{\dv S}{c_W} \left(\frac{(\tr S^2)^{\alpha/2}}{m|\cos(\pi\alpha/2)|} +\frac{K_2}{m^{2-3\varepsilon}}\right)\exp(-m_H(\tr S^2)^{\alpha/2}+K_1)\\
	\le&\frac{1}{m} \int_{W}\frac{\dv S}{c_W} \left(\frac{2(\tr S^2)^{\alpha/2}}{|\cos(\pi\alpha/2)|} +K_2\right)\exp(-m_H(\tr S^2)^{\alpha/2}+K_1).
	\end{split}
	\end{equation}
	In the second inequality we exploited $0<\varepsilon<1/4$ and $m\ge 1$. The integral is  $m$-independent and absolutely integrable such that it yields a constant only depending on the spectral measure $H$ and the stability  parameter $\alpha$.
	
	Very similarly we compute for $\alpha=1$
	\begin{equation}
	\begin{split}
	I_1\le& \int_{\tr S^2<T^2}\frac{\dv S}{c_W} \left(\frac{1}{m}\sqrt{\tr S^2}\sqrt{1+\left(\frac{\log(\tr S^2)}{\pi}\right)^2 }+\frac{\varepsilon^3K_2\log^3(m)}{m^{2-3\varepsilon}}\right)\exp\left(-m_H\sqrt{\tr S^2}+K_1\right)\\
	\overset{0<\varepsilon<1/4}{\le}&\frac{1}{m} \int_{W}\frac{\dv S}{c_W} \left(\sqrt{\tr S^2}\sqrt{1+\left(\frac{\log(\tr S^2)}{\pi}\right)^2 } +2K_2\right)\exp\left(-m_H\sqrt{\tr S^2}+K_1\right),
	\end{split}
	\end{equation}
	where we have used that $\log^3(m)/(64 m^{1/4})<2$ for any $m\ge 1$. Also this integral is absolutely convergent and $m$-independent. This closes the proof.
\end{proof}

\section{Proof of Theorem~\ref{thm2} statement (3)}\label{s.proof3}

The strategy of proving the third statement of Theorem~\ref{thm2} is the following. For each $\alpha$ one wants to find  
\begin{enumerate}
	\item a specific spectral measure $H$,
	\item and a specific $m$-independent test function $\varphi: W\mapsto\mathbb{C}$  which is bounded and absolutely integrable,
\end{enumerate} 
such that one can show an expansion as follows
\begin{equation}\label{eq.4.1}
\left|\int_W dX \varphi(X)(p_m(X)-p_\infty(X))\right|=\frac{c}{m}+o(m^{-1})
\end{equation}
for some $c>0$ in the limit $m\to\infty$.  Now by the following inequality
\begin{equation}
\left|\int_W dX \varphi(X)(p_m(X)-p_\infty(X))\right|\leq\sup_{Y\in W}|p_m(Y)-p_\infty(Y)| \int_W dX |\varphi(X)|,
\end{equation}
together with the fact that $ \int_W dX |\varphi(X)|$ is assumed to be finite and non-zero, one is able to give a lower bound for the supremum on the right hand side. Certainly, it is sufficient to find only one spectral measure and one test function satisfying these conditions, as then the supremum is bounded from below also by the rate of convergence $1/m$. 

For the particular spectral measure $H_\eig$ on $\mathbb{S}^{N-1}$ from which we can construct $H$ on $\mathbb{S}_{\Herm(N)}$ via the conjugation $R=U\diag(r)U^\dagger$ with a Haar distributed unitary matrix $U\in\U(N)$, we choose the orbital measure
\begin{equation}
	H_{\eig}^{(t)}(\dv r)=\frac{1}{2N}\sum_{j=1}^N[t\,\delta(r-e^{(j)})+(1-t)\delta(r+e^{(j)})]\dv r\qquad{\rm with }\ t\in[0,1],
\end{equation}
where $e^{(j)}$ is a vector with zero entries everywhere except at its $j$-th entry where it is $1$. In other words, the corresponding spectral measure $H^{(t)}$ is the convex linear combination of two uniform measures on the $\U(N)$-orbits of the rank-one matrices $\diag(\pm1,0,\ldots,0)$.  Since the trace of $U\diag(\pm1,0,\ldots,0)U^\dagger$ is not vanishing, the support of the corresponding invariant stable distribution will be $W=\Herm(N)$.

With this spectral measure we are able to state the following lemma, which is equivalent to Theorem~\ref{thm2} statement (3).

\begin{lemma}[Theorem~\ref{thm2} statement (3)]\label{lb.lemma}
	Let $Z_j$ be independent and identical copies of the random matrix corresponding to $S(\alpha,H^{(t)}_\eig,0)$, and $U_j\in\U(N)$ be independent Haar distributed random matrices. We denote by $q_m$ the distribution of the quantity
	\begin{equation}
		Y_m=\frac{1}{m^{1/\alpha}}\sum_{j=1}^mU_jZ_jU_j^\dagger-\frac{2t_0\log m}{\pi}I_N\ind_{\alpha=1}
	\end{equation}
	and $q_\infty$ the distribution of $S(\alpha,H^{(0)},0)$. Then
	\begin{equation}\label{lb.lemma.eqn}
		\sup_{X\in\Herm(N)}|q_m(X)-q_\infty(X)|=O(1/m).
	\end{equation}
\end{lemma}

In the rest of this section we will firstly give a specific choice for the test function and give its properties  in Lemma~\ref{varphi}, and then show in Lemma~\ref{lb.lemma1} that ~\eqref{eq.4.1} is true for some $c\ge 0$. Finally in Lemma~\ref{lb.lemma2} we show that the case $c=0$ can be excluded with suitable choices of the parameter $t$ in the spectral measure and the test function.

\subsection{The test function $\varphi_N^{(\varepsilon)}$}
For the test functions we choose the unitarily invariant family
\begin{equation}
\varphi_{N}^{(\varepsilon)}(X):=\frac{\Delta(x)}{N!\varepsilon^{(N-1)(N-2)}}\Delta(-\partial_x)\,{\rm perm}[\zeta_{\varepsilon,1}(x_a),\overbrace{\zeta_{\varepsilon,\varepsilon}(x_a),\ldots,\zeta_{\varepsilon,\varepsilon}(x_a)}^{N-1\ {\rm times}}]_{a=1,\ldots,N}
\end{equation}
with
\begin{equation}
\zeta_{\varepsilon,\kappa}(x):=\frac{\pi^{2N} e^{-i\kappa x}\cos(\varepsilon x)}{\prod_{k=1}^N\left(\pi^2-4\varepsilon^2x^2/(2k-1)^2\right)},
\end{equation}
where $x_1,\ldots,x_N$ are the eigenvalues of $X$, ${\rm perm}$ is the permanent, and we have used the Vandermonde determinant
\begin{equation}
\Delta(x)=\det[x_a^{b-1}]_{a,b=1,\ldots,N}=\prod_{1\leq a<b\leq N}(x_b-x_a).
\end{equation}
The Vandermonde determinant of the partial derivatives $\Delta(-\partial_x)$ is defined similarly via replacing $x_a$ by $-\partial_{x_a}$. Note $\zeta_{\varepsilon,\kappa}$ is only a slight modification of the regularisation function used in~\cite[Eq.~(2.41)]{KK16}. Its properties are also discussed in the same paper while we restate them here.

\begin{lemma}\label{varphi}
	The function $\varphi_N^{(\varepsilon)}$ has the following properties:
	\begin{enumerate}
		\item it is bounded, continuous, and absolutely integrable;
		\item  its Fourier transform is equal to
		\begin{equation}
		\hat\varphi_{N}^{(\varepsilon)}(S)=\frac{1}{N!\varepsilon^{(N-1)(N-2)}}{\rm perm}\bigl[\hat\zeta_{\varepsilon,1}(s_a),\overbrace{\hat\zeta_{\varepsilon,\varepsilon}(s_a),\ldots,\hat\zeta_{\varepsilon,\varepsilon}(s_a)}^{N-1\ {\rm times}}\bigl]_{a=1,\ldots,N}\geq0
		\end{equation}
		where $(s_1,\ldots,s_N)\in\mathbb{R}$ are the eigenvalues of $S\in\Herm(N)$ and we have employed
		\begin{equation}\label{hatzeta}
		\hat\zeta_{\varepsilon,\kappa}(s)=\frac{2\pi c}{\varepsilon}\cos^{2N-1}\left(\frac{\pi}{2\varepsilon}(s-\kappa)\right)\ind_{s\in(\kappa-\varepsilon,\kappa+\varepsilon)}\quad {\rm with}\quad c^{-1}=\int_{-1}^1\dv u\cos^{2N-1}\left(\frac{\pi}{2}u\right);
		\end{equation}
		\item	 the integral of a bounded, continuous and unitarily invariant function $f:\Herm(N)\mapsto\mathbb{C}$ against the  Fourier transform $\hat\varphi_{N}^{(\varepsilon)}$ satisfies the limit
					\begin{equation}
					\begin{split}
					\lim_{\varepsilon\to0}\int_{\Herm(N)}\dv S f(S)\hat\varphi_{N}^{(\varepsilon)}(S)=&\tilde{c} f(\diag(1,0,\ldots,0))
					\end{split}
					\end{equation}
					with the positive constant
					\begin{equation}\label{tilde_c}
					\tilde{c}=\left(\prod_{j=1}^N\frac{\pi^{j-1}}{j!}\right)\int_{(-1,1)^N} \dv s_1\cdots\dv s_N \Delta^2(s_2,\ldots,s_N)\prod_{j=1}^N 2\pi c \,\cos^{2N-1}\left(\frac{\pi}{2}s_j\right)>0.
					\end{equation}
	\end{enumerate}
\end{lemma}
\begin{proof}
	To see that $\varphi_N^{(\varepsilon)}$ is  bounded and continuous we notice that the $2N$ singularities at $|x|=|2k-1|\pi/(2\varepsilon)$ for $k=1,\ldots,N$ are removable. Furthermore the numerator of $\zeta_{\varepsilon,\kappa}(x)$ is bounded by $\pi^{2N}$ while the denominator grows like $x^{2N}$ when $|x|\to\infty$.
	
	To check the absolute integrability, we can expand the two determinants and the permanent to obtain
	\begin{equation}
	\begin{split}
	\int_{\Herm(N)}|\varphi_N(X)|\dv X\le& \frac{1}{N!\varepsilon^{(N-1)(N-2)}}\sum_{\sigma,\rho,\tau\in S_N}\int_{-\infty}^\infty \dv x_{\tau(1)}\left|x_{\tau(1)}^{\sigma(1)-1}\partial_{x_{\tau(1)}}^{\rho(1)-1}\zeta_{\varepsilon,1}(x_{\tau(1)})\right|\\
	&\times\prod_{j=2}^N\int_{-\infty}^\infty \dv x_{\tau(j)}\left|x_{\tau(j)}^{\sigma(j)-1}\partial_{x_{\tau(j)}}^{\rho(j)-1}\zeta_{\varepsilon,\varepsilon}(x_{\tau(j)})\right|.
	\end{split}
	\end{equation}
	Hence, it is sufficient to show that each factor in the summands converges, i.e.
	\begin{equation}
		\int_{-\infty}^\infty \dv x\left|x^{j-1}\partial_{x}^{k-1}\zeta_{\varepsilon,\kappa}(x)\right|=8\pi\varepsilon\int_0^\infty \dv x\, x^{j-1}\left|\partial_{x}^{k-1} \frac{e^{-i\kappa x}\cos(\varepsilon x)}{\prod_{k=1}^N\left(\pi^2-\frac{4\varepsilon^2x^2}{(2k-1)^2}\right)}\right|<\infty
	\end{equation}
	for all $j,k=1,\ldots,N$ and $a\in\{1,\varepsilon\}$. The integrabilty is, however, immediate as the integrand on the right hand side can be bound by $\tilde{c} x^{j-1}/(1+x^2)^N$ for a suitably large $\tilde{c}>0$. This can be seen from the continuity of the integrand on $\mathbb{R}_+$. Namely up to $x< \pi N /\varepsilon$ we can bound the function by $\tilde{c}_<x^{j-1}$ and for $x\geq \pi N /\varepsilon$ the factor $1/\prod_{k=1}^N\left|\pi^2-4\varepsilon^2x^2/(2k-1)^2\right|$ is strictly monotonically decreasing so that the dominating contribution for $x\to\infty$ is gained when acting the  derivatives on the exponential and trigonometric function which gives the bound $\tilde{c}_\geq/x^{2N-j+1}$ for the integrand in this regime.
	
	Statement (2) can be obtained by a simple calculation of the Fourier transform of $\zeta$. By the derivative principle~\cite[Proposition 4]{KZ20}, the matrix Fourier transforms of $\varphi_N^{(\varepsilon)}$ on $\Herm(N)$  is equal to the multivariate Fourier transform of the permanent on $\mathbb{R}^N$, i.e.,
	\begin{equation}
	\begin{split}
	\hat\varphi_{N}^{(\varepsilon)}(S)=&\frac{1}{N!\varepsilon^{(N-1)(N-2)}}\int_{\mathbb{R}^N} e^{is_1x_1}\dv x_1\cdots  e^{is_Nx_N}\dv x_N\,{\rm perm}[\zeta_{\varepsilon,1}(x_a),\zeta_{\varepsilon,\varepsilon}(x_a),\ldots,\zeta_{\varepsilon,\varepsilon}(x_a)]_{a=1,\ldots,N}\\
	=&\frac{1}{N!\varepsilon^{(N-1)(N-2)}}\,{\rm perm}\left[\int_{-\infty}^\infty e^{is_1x}\dv x\zeta_{\varepsilon,1}(x),\int_{-\infty}^\infty e^{is_2x}\dv x\zeta_{\varepsilon,\varepsilon}(x),\ldots,\int_{-\infty}^\infty e^{is_Nx}\dv x\zeta_{\varepsilon,\varepsilon}(x)\right]_{a=1,\ldots,N}.
	\end{split}
	\end{equation}
	The one-fold integrals are the Fourier transforms of $\zeta_{\varepsilon,\kappa}$. It is actually simpler to check that the inverse Fourier transform of~\eqref{hatzeta} yields $\zeta$. Namely, one applies the binomial formula for $\cos^{2N-1}[\pi(s-\kappa)/(2\varepsilon)]$ and integrates against the interval $[\kappa-\varepsilon,\kappa+\varepsilon]$and integrates term by term, 
	\begin{equation}
	\begin{split}
	\int_{-\infty}^\infty \frac{e^{-isx}\dv s}{2\pi}\hat{\zeta}_{\varepsilon,\varepsilon}(s)=&\frac{\varepsilon ce^{-ix\kappa}}{2^{2N-2}}\sum_{j=0}^{2N-1}\binom{2N-1}{j}\frac{\sin(\pi(2N-1-2j)/2-\varepsilon x)}{\pi(2N-1-2j)/2-\varepsilon x}\\
	=&\frac{\pi\varepsilon ce^{-ix\kappa}}{2^{2N-2}}\cos(\varepsilon x)\sum_{j=1}^{N}\binom{2N-1}{N-j}\frac{(-1)^{j}(2j-1)}{\pi^2(2j-1)^2/4-\varepsilon^2 x^2}.
	\end{split}
	\end{equation}
	 Once all $2N$ contributions are resummed, one regains $\zeta_{\varepsilon,\varepsilon}$ up to a constant. The constant can be actually fixed by noticing that $\zeta_{\varepsilon,\varepsilon}(0)=1$.
	 
	 The non-negativity of the Fourier transform is a direct consequence that $\cos(x)>0$ when $x\in(-\pi/2,\pi/2)$.
	 
	 The weak limit can be obtained via Lebesgue's dominated convergence theorem. We choose an arbitrary bounded, continuous and unitarily invariant function $f:\Herm(N)\mapsto\mathbb{C}$. For this reason we diagonalise $S=\tilde{U}\diag(s_1,\ldots,s_N)\tilde{U}^\dagger$ which leads to
					\begin{equation}
					\begin{split}
					\int_{\Herm(N)}\dv S f(S)\hat\varphi_{N}^{(\varepsilon)}(S)=&\left(\prod_{j=1}^N\frac{\pi^{j-1}}{j!}\right)\int_{\mathbb{R}^N}\dv s_1\cdots\dv s_N\Delta^2(s)f(s)\hat\varphi_{N}^{(\varepsilon)}(s)\\
					=&\left(\prod_{j=1}^N\frac{\pi^{j-1}}{j!}\right)\int_{\mathbb{R}^N}\frac{\dv s_1}{\varepsilon}\cdots\frac{\dv s_N}{\varepsilon}\frac{\Delta^2(s)}{\varepsilon^{(N-1)(N-2)}}f(s)\\
					&\times2\pi c\,\cos^{2N-1}\left(\frac{\pi}{2\varepsilon}(s_1-1)\right)\ind_{s_1\in(1-\varepsilon,1+\varepsilon)}\prod_{j=2}^N2\pi c\,\cos^{2N-1}\left(\frac{\pi}{2\varepsilon}(s_j-\varepsilon)\right)\ind_{s_j\in(0,2\varepsilon)}.
					\end{split}
					\end{equation}
Next we substitute $s_1\to1+\varepsilon s_1$ and $s_j\to \varepsilon+\varepsilon s_j$ for all $j=2,\ldots,N$ for which we employ
					\begin{equation}
					\Delta^2(1+\varepsilon s_1,\varepsilon+\varepsilon s_2,\ldots,\varepsilon+\varepsilon s_N)=\varepsilon^{(N-1)(N-2)}\Delta^2(s_2,\ldots,s_N)\prod_{j=2}^N(\varepsilon+\varepsilon s_j-1-\varepsilon s_1)^2.
					\end{equation}
This integral simplifies as follows
					\begin{equation}
					\begin{split}
					\int_{\Herm(N)}\dv S f(S)\hat\varphi_{N}^{(\varepsilon)}(S)=&\left(\prod_{j=1}^N\frac{\pi^{j-1}}{j!}\right)\int_{(-1,1)^N}\hspace*{-0.5cm}\dv s_1\cdots\dv s_N \Delta^2(s_2,\ldots, s_N)f(\diag(1+\varepsilon s_1,\varepsilon+\varepsilon s_2,\ldots,\varepsilon+\varepsilon s_N))\\
					&\times2\pi c\,\cos^{2N-1}\left(\frac{\pi s_1}{2}\right)\prod_{j=2}^N2\pi c\,\cos^{2N-1}\left(\frac{\pi s_j}{2}\right)(\varepsilon+\varepsilon s_j-1-\varepsilon s_1)^2.
					\end{split}
					\end{equation}
The integrand is bounded and continuous and, therefore, has a well-defined point-wise limit. Since the integration domain is bounded, we can apply  Lebesgue's dominated convergence theorem yielding the last claim of Lemma~\ref{varphi}.				
\end{proof}

One property that motivates us to use this test function is the support of the eigenvalues of its Fourier transform. It is concentrated on the set of the fixed vectors $e^{(1)},\ldots,e^{(N)}$ in the weak limit, where the variable $\varepsilon$ says how close it is to this set. The benefit is that one can show the leading order term of the $S$-integral to be non-vanishing when $\varepsilon\to0$. Then, the continuity for $\varepsilon>0$  will allow us to conclude that it has to be also non-vanishing for some non-zero $\varepsilon$.

\subsection{Showing Eq.~\eqref{eq.4.1} with $c\ge 0$}

The main idea of the proof for Lemma~\ref{lb.lemma} is to make an estimate regarding the difference between the two characteristic functions in Lemma~\ref{lb.lemma1}. Then, one can argue that the dominating term of~\eqref{lb.lemma.eqn} should be bounded from below by the dominating term obtained for the characteristic functions.

The following lemma is a more detailed version of Eq.~\eqref{eq.4.1} with $c\ge 0$.

\begin{lemma}\label{lb.lemma1}
	With the general setting as in Theorem~\ref{thm2}, one has
	\begin{equation}\label{lb.lemma1.eqn}
	\int_{\Herm(N)}\dv S \,\varphi_N^{(\varepsilon)}(S)( q_m(S)- q_\infty(S))=\frac{1}{2m}\int_{\Herm(N)}\dv S \,\hat\varphi_N^{(\varepsilon)}(S)\exp(-\la w_\alpha(USU^\dagger)\ra) \, v(S)+O(m^{-2}),
	\end{equation}
	where $v(S)$ is the variance~\eqref{def.variance}.
\end{lemma}

\begin{proof}
	Since both $\varphi_N^{(\varepsilon)}$ and $(q_m-q_\infty)$ are bounded, continuous and absolutely integrable on $\Herm(N)$ they are also square integrable and we can apply Plancherel's theorem,
	\begin{equation}\label{pr.thm.2.b}
		\int_{\Herm(N)} \varphi_N^{(\varepsilon)}(X)(q_m(X)-q_\infty(X))\dv X=\int_{\Herm(N)}\hat\varphi_N^{(\varepsilon)}(S)(\hat q_m(S)-\hat q_\infty(S))\dv S.
	\end{equation}
	The remaining task is to prove that the right hand side of~\eqref{pr.thm.2.b} behaves like the right side of~\eqref{lb.lemma1.eqn} when $m$ is large.

	We recall that the support of $\hat\varphi_{N}^{(\varepsilon)}(S)$ is the subset of the unit sphere restricted by the inequality
	\begin{equation}
	\rho_{\rm l}^2:=1-\varepsilon\leq\tr S^2\leq1+(2N-1)\varepsilon=:\rho_{\rm u}^2
	\end{equation}
	as give in Lemma~\ref{varphi}. This restriction tells us that for suitably large $m$ the modulus of $D$, which is the difference~\eqref{D}, can be brought below $1/2$. Hence, we exploit the Taylor expansion of the exponential function
	\begin{equation}
	\begin{split}
	\hat q_m(S)-\hat q_\infty(S)&=\exp(-\la w_\alpha(USU^\dagger)\ra)\left( e^D-1\right)= \exp(-\la w_\alpha(USU^\dagger)\ra)\left(\frac{v(S)}{2m}+R_{e,2}(D)\right)
	\end{split}
	\end{equation}
	where
	\begin{equation}
	|R_{e,2}(D)|\leq e^{1/2} |D|^2/2\leq k/m^2
	\end{equation} 
	with a suitable constant $k>0$; see~\eqref{exp.Taylor} and Lemma~\ref{lemma_difference} with $T^2=1+(2N-1)\varepsilon$.
	
	The integral for the leading order $v(S)/2m$ is exactly the first term on the right hand side of~\eqref{lb.lemma1.eqn}. The integral over the remainder times $m^2$ is actually bounded by a constant independent of $m$ because of
	\begin{equation}
		\begin{split}
	&\left|\int_{\Herm(N)}\dv S\, \hat\varphi_N^{(\varepsilon)}(S) R_{e,2}(D)\exp(-\la w_\alpha(USU^\dagger)\ra)\right|\\
	&\quad\quad\leq \int_{\Herm(N)}\dv S\left| \hat\varphi_N^{(\varepsilon)}(S) R_{e,2}(D)\exp(-\la w_\alpha(USU^\dagger)\ra)\right|\leq\frac{k}{m^2} \int_{\rho_{\rm l}^2\leq\tr S^2\leq\rho_{\rm u}^2}\dv S,
	\end{split}
	\end{equation}
	where we have employed the boundedness of $\hat\varphi_N^{(\varepsilon)}(S)$ and $\la \re w_\alpha(USU^\dagger)\ra\geq0$ for all $S\in\Herm(N)$, which is claimed in Lemma~\ref{lem:walpha-bounds}.
\end{proof}

Now we are almost in the position to prove Lemma~\ref{lb.lemma}. Yet there is still a technical hurdle which needs to be overcome. When integrating the test function $\varphi_N^{(\varepsilon)}$ against the difference $q_m-q_\infty$, it is evidently
\begin{equation}\label{pr.thm.2.a}
	\begin{split}
		\left|\int_{\Herm(N)}\dv X \varphi_N^{(\varepsilon)}(X)(q_m(X)-q_\infty(X)) \right|\le \sup_{X\in\Herm(N)}|q_m(X)-q_\infty(X)|\int_{\Herm(N)} \dv X|\varphi_N^{(\varepsilon)}(X)|,
	\end{split}
\end{equation}
which suggests us to give a lower bound of the left hand side. Moreover, Lemma~\ref{lb.lemma1} shows that the left hand side of~\eqref{pr.thm.2.a} has a dominating term of the form $cm^{-1}$ with $c\ge 0$ independent of $m$. However, it is not guaranteed so far that the constant $c$ of this dominating term is not vanishing. To show that exactly this is the case  needs to be our next and final achievement for proving the third statement of Theorem~\ref{thm2}.

\subsection{Excluding the case $c=0$ and finalising the proof}

As a first step to exclude the case $c=0$, we notice that for the first column $u=(u_{1,1},\ldots,u_{N,1})^\top$ of a Haar distributed unitary matrix $U\in\U(N)$ the following two maps,
\begin{equation}\label{lb.lemma2a}
		u\mapsto\sum_{j=1}^N|u_{j,1}|^{2\alpha}\quad {\rm for}\ \alpha\neq1\qquad{\rm and}\qquad 	u\mapsto\sum_{j=1}^N|u_{j,1}|^2\log|u_{j,1}|,
\end{equation}
are not constant random variables when $N>1$, meaning each of them does not take only a single value with probability one. Actually, they have both bounded and continuous probability densities.
A proof of this statement directly follows from an explicit parametrisation of the vector $u=(u_{1,1},\ldots,u_{N,1})^\top$ and computing the result. One important consequence for our purpose is that these random variables have a non-vanishing variance.

\begin{lemma}\label{lb.lemma2}
	Employing the spectral measure $H^{(t)}_\eig$ and the test function $\varphi_N^{(\varepsilon)}$, for all $
	\alpha\in(0,2]$ and $N\ge 2$, there exists an $\varepsilon>0$ and a $t>0$ such that
	\begin{equation}
	\int_{\Herm(N)}\dv S \,\hat\varphi_N^{(\varepsilon)}(S)\exp(-\la w_\alpha(USU^\dagger)\ra) \, v(S)\neq0.
	\end{equation}
\end{lemma}

\begin{proof}
	We denote the integral as follows
	\begin{equation}
	I_\varepsilon:=\int_{\Herm(N)}\dv S \,\hat\varphi_N^{(\varepsilon)}(S)\exp(-\la w_\alpha(USU^\dagger)\ra) \, v(S).
	\end{equation}
	Due to Lemma~\ref{walpha} we know that the factor $\exp(-\la w_\alpha(USU^\dagger)\ra) \, v(S)$ is bounded and continuous, and it is unitarily invariant since we average over $U$. Thus, this factor satisfies the conditions of the third statement in Lemma~\ref{varphi}. 
	Hence, to prove Lemma~\ref{lb.lemma2} it is sufficient to show that
	\begin{equation}
	\lim_{\varepsilon\to0}I_\varepsilon=\tilde{c}\exp(-\la w_\alpha(US_0U^\dagger)\ra) \, v(S_0)\neq0
	\end{equation}
	with $S_0=\diag(1,0,\ldots,0)$ and the positive constant $\tilde c$ given in~\eqref{tilde_c}, because the continuity in $\varepsilon>0$ (guaranteed by Lebesgue's dominated convergence theorem) implies $I_\varepsilon\neq0$ for some $\varepsilon>0$ being close to this limit. 
	
	The exponential function $\exp(-\la w_\alpha(US_0U^\dagger)\ra)$ is surely non-zero since $\la w_\alpha(US_0U^\dagger)\ra$ is bounded, see Lemma~\ref{walpha}. Thence, the whole problem reduces to show that $v(S_0)$ is non-vanishing. 
	Before showing this, we would like to introduce the notations
	\begin{equation}
	A(S):=\re w_\alpha(USU^\dagger)\qquad{\rm and}\qquad B(S):=\im w_\alpha(USU^\dagger).
	\end{equation}
	In general, for a complex random variable $Z=A+iB$, its complex variance $\la (Z-\la Z\ra)^2\ra$ can be expanded as
	\begin{equation}\label{complex_variance}
	\la (Z-\la Z\ra)^2\ra=2\left(\la (A-\la A\ra)^2\ra-\la (B-\la B\ra)^2\ra\right)+4i\left(\la AB\ra-\la A\ra\la B\ra\right).
	\end{equation}
	This means that a vanishing variance $v(S_0)=0$ implies that the variances of $A$ and $B$ must agree and their covariance must vanish, i.e.,
	\begin{equation}\label{complex_variance.b}
	\la (A-\la A\ra)^2\ra=\la (B-\la B\ra)^2\ra\qquad{\rm and}\qquad \la AB\ra=\la A\ra\la B\ra.
	\end{equation}
	
	Let us consider the case $\alpha\ne 1$, first. With the particular spectral measure $H^{(t)}_\eig$ and $t\in[0,1]$, the function $w_\alpha$ simplifies to
	\begin{equation}
	\begin{split}
	w_\alpha(US_0U^\dagger)
	&=\frac{1}{2N}\sum_{j=1}^N\left[t\,\nu_\alpha\left(|u_{j,1}|^2\right)+(1-t)\nu_\alpha\left(-|u_{j,1}|^2\right)\right]\\
	&=\frac{1}{N}\sum_{j=1}^N|u_{j,1}|^{2\alpha}\left(1-i(2t-1)\tan\frac{\pi\alpha}{2}\right).
	\end{split}
	\end{equation}
	As $A(S_0)$ is a random variable $\sum_{j=1}^N|u_{j,1}|^{2\alpha}$ with a non-zero variance, we can choose such a $t$ such that $(1-2t)^2\tan[\pi\alpha/2]\neq1$ for a given $\alpha$ such that $B(S_0)$ is equal to $A(S_0)$ up to the factor $(1-2t)\tan[\pi\alpha/2]$. Now $B(S_0)$ has a different variance from the variance for $A(S_0)$, which therefore leads to a non-vanishing real part for $v(S_0)$ and proves our claim in the case $(0,1)\cup(1,2]$.
	
	For $\alpha=1$ we have
	\begin{equation}
	\begin{split}
		w_\alpha(US_0U^\dagger)&=\frac{1}{N}\sum_{j=1}^N\left(|u_{j,1}|^2+\frac{2i}{\pi}(2t-1)|u_{j,1}|^2\log\left||u_{j,1}|^2\right|\right)\\
		&=\frac{1}{N}\left(1+\frac{4i}{\pi}(2t-1)\sum_{j=1}^N|u_{j,1}|^2\log|u_{j,1}|\right).
	\end{split}
	\end{equation}	
	The real part $A(S_0)=1/N$ is deterministic in $U$ and therefore has a vanishing variance. However, by setting $2t-1\ne 0$ the imaginary part $B(S_0)$ is proportional to the random variable $\sum_{j=1}^N|u_{j,1}|^2\log|u_{j,1}|$, which has a non-vanishing variance with respect to the Haar measure on $\U(N)$. Therefore, anew the real part of $v(S_0)$ is non-vanishing.
	
	Therefore, for every stability exponent $\alpha\in(0,2]$ we find a $t\in[0,1]$ such that $v(S_0)\neq0$. This implies $\lim_{\varepsilon\to0} I_\varepsilon\neq0$ and the continuity in $\varepsilon>0$ results in an existence of an $\varepsilon>0$ such that $I_\varepsilon\neq0$ which we wanted to show.
\end{proof}

\begin{proof}[Proof of Theorem~\ref{thm2} statement (3)]
Choosing the test function $\varphi_N^{(\varepsilon)}(S)$ with a suitable $\varepsilon>0$ and the spectral measure $H_\eig^{(t)}$ with a suitable $t\in[0,1]$ we have shown in Lemma~\ref{lb.lemma1} that
	\begin{equation}
	\int_{\Herm(N)}\dv S \,\hat\varphi_N^{(\varepsilon)}(S)(\hat q_m(S)-\hat q_\infty(S))=\frac{1}{2m}\int_{\Herm(N)}\dv S \,\hat\varphi_N^{(\varepsilon)}(S)\exp(-\la w_\alpha(USU^\dagger)\ra) \, v(S)+O(m^{-2})
	\end{equation}
	when $m\gg1$, and in Lemma~\ref{lb.lemma2} we have proven that the integral in the leading term is non-vanishing, i.e.,
	\begin{equation}
	I_\varepsilon=\int_{\Herm(N)}\dv S \,\hat\varphi_N^{(\varepsilon)}(S)\exp(-\la w_\alpha(USU^\dagger)\ra) \, v(S)\neq0.
	\end{equation}
	We combine this with Eqs.~\eqref{pr.thm.2.a} as well as the second statement of Theorem~\ref{thm2} to find
	\begin{equation}
	 \frac{I_\varepsilon/m+O(m^{-2})}{2\int_{\Herm(N)} \dv X|\varphi_N^{(\varepsilon)}(X)|}\leq\sup_{X\in\Herm(N)}|q_m(X)-q_\infty(X)|\leq\frac{C}{m}
	\end{equation}
	for some constant $C>0$. This inequality shows the optimality of the rate of convergence.
\end{proof}

\section{Concluding remarks}\label{sec:conclusio}

We have shown that the construction of any invariant stable random matrix ensemble on $\Herm(N)$ can be realised as in Theorem~\ref{thm2}. Hence, the problem has been reduced from a stable random vector in the $N^2$-dim space $\Herm(N)$ to  the generation of a stable random vector in $\mathbb{R}^N$. For the latter one, one can exploit the methods by Davydov and Nagaev~\cite{DN02}. The stability exponent $\alpha$ and the spectral measure serve as an input in this kind of generation.

Furthermore, we have also provided a rate of convergence which is equal to $1/m$, and we have proven that this rate is optimal in the set of all invariant stable random matrix ensembles with a fixed stability exponent $\alpha$. We reckon that this rate of convergence is optimal even for a fixed spectral measures $H$  as long as the span of the support of $H$ is higher than one-dimensional, meaning the random matrices are not proportional to the identity matrix. The reason may lie in the fact that the independent and identically Haar distributed unitary matrices that are conjugated to the diagonal stable random matrices mixes those spectra. This mixture seems to be purely determined by the unitary matrices and the number of copies added but less by the chosen distribution of the diagonal random matrices. Indeed most of the lemmas and statements carry over to  fixed  arbitrary spectral measures. Only the arguments in the proof of Lemma~\ref{lb.lemma2} must be generalised to an arbitrary spectral measure $H$ for which the chosen test function does not always work as one can readily see, choosing the parameter $t$ wrongly leads to a vanishing variance.

We strongly believe that the rate of convergence for our construction is in general of order $1/m$, provided that the matrices have Haar distributed eigenvectors. Even the most harmless case $\alpha=2$ which is the Gaussian satisfies this rate. Indeed, for the case $N=2$ and the traceless setting (the tracial part remains always a Gaussian) we find this rate. Then, the random matrix is
\begin{equation}
Y_m=\frac{1}{\sqrt{m}}\sum_{j=1}^m x_j U_j\diag(1,-1)U_j^\dagger,
\end{equation}
$U_1,\ldots,U_m\in\U(2)$ independently Haar distributed and $x_1,\ldots,x_M\in\mathbb{R}$ independently normal distributed, i.e., they have the joint distribution
\begin{equation}
P(x_1,\ldots,x_m)=\prod_{j=1}^m\frac{1}{\sqrt{2\pi}}e^{-x_j^2/2}.
\end{equation}
Then, the characteristic function of $Y_m$ is
\begin{equation}
\hat{p}_m(S)=\left(\int_{-\infty}^{\infty} \frac{\sin(2xs/\sqrt{m})}{2xs/\sqrt{m}}e^{-x^2/2}\frac{dx}{\sqrt{2\pi}}\right)^m=\left(\sqrt{\frac{\pi}{8}}\frac{{\rm erf}(\sqrt{2}s/\sqrt{m})}{s/\sqrt{m}}\right)^m
\end{equation}
with $s\in\mathbb{R}$ and $-s$ the two eigenvalues of $S\in\Herm_0(2)$ and ${\rm erf}$ being the error function. The Taylor expansion of the characteristic function for large $m$ yields a $1/m$ series and $\hat{p}_m(S)$ is uniformly bounded in $m>2$ by $K/(1+s^2)$ with an $m$ independent constant $K$. Therefore, also the expansion of the inverse Fourier transform has an expansion in $1/m$. Note that this is not in contradiction to the Berry-Esseen theorem, because the latter only gives an upper bound of the rate of convergence for a much broader class of ensembles without restricting the eigenvectors to be Haar distributed.

A direct consequence of Theorem~\ref{thm2} is a rate of the total variation distance. This follows along the same ideas of the proof of~\cite[Thm 3.2]{DN02} by making use of~\cite[Thm 3.3]{DN02}. As we have a different rate for the supremum norm, a different rate for the total variation distance should result.

\begin{corollary}\label{col3}
	Let $\sigma(W)$ be the Borel $\sigma$-algebra of $W$. With the above settings, the total variation distance can be estimated by
	\begin{equation}
		\sup_{A\in\sigma( W)}\left|\int_Ap_m(X)\dv X-\int_Ap_\infty(X)\dv X\right|=O(m^{-\alpha/(N^2+\alpha)}).
	\end{equation}
\end{corollary}

In our case where both probability density functions $p_m$ and $p_\infty$ for the compared measures exist, the total variation distance can be rewritten as the $L^1$-distance between those functions
\begin{equation}
	\frac{1}{2}\int_{W} \left| p_m(X)-p(X)\right|\dv X=O(m^{-\alpha/(N^2+\alpha)}).
\end{equation}
Compared with the rate $-\min(\alpha,2-\alpha)/(N^2+\alpha)$ obtained in~\cite{DN02}, our rate of convergence is the same when $\alpha\le 1$ but evidently sharper when $\alpha>1$. As conjectured in~\cite{DN02}, we also conjecture that this rate of convergence can be further improved to $O(1/m)$ as in Theorem~\ref{thm2}, though it is not clear how one can justifiy it. 

We make a final remark regarding another sampling procedure studied in~\cite{BNR93,DN02}.

\begin{remark}[]
	In~\cite{BNR93} another sampling procedure has been consider, namely approximating a stable random vector by some other stable random vectors but with slightly different spectral measure. In particular, a discrete measures has been used as an approximation. Then it was shown that it is always possible to find a discrete measure to approximate the desired spectral measure, such that the supremum distance of stable distributions is under control. In~\cite{DN02} a further detailed analysis is given as a bound for the supremum distance of the stable distributions by the Prokhorov distance between the two spectral measures, where the latter one is induced by the weak topology of probability measures. In our case, a further research is to investigate how one can establish a similar bound by the eigenvalue part of the spectral measures, which requires more detailed studies on how the Prokhoriv distance (or any equivalent distances) changes while integrating over all the eigenvectors.
\end{remark}

\section*{Acknowledgements}

MK acknowledges financial support from the Australian Research Council of the Discovery Project grant DP210102887. JZ acknowledges the support from FWO Flanders project EOS  30889451.


\, 
\end{document}